\documentclass[12pt]{amsart}
\usepackage{amsmath,mathtools,amsthm,amsfonts,bigints,amssymb, mathrsfs, tikz-cd, xcolor, float}
\usepackage{tikz}
\usetikzlibrary{patterns}
\usepackage{pgf,tikz}

\usepackage{pgfplots}
\usepgfplotslibrary{fillbetween}

\usetikzlibrary{arrows}

\newtheorem{theorem}{Theorem}[section]
\newtheorem{lemma}[theorem]{Lemma}

\newtheorem{corollary}[theorem]{Corollary}

\theoremstyle{definition}
\newtheorem{definition}[theorem]{Definition}

\newtheorem{remark}[theorem]{Remark}

\newcommand{\what}{\widehat}

\newcommand{\R}{\mathbb R}%
\newcommand{\C}{\mathbb C}%
\newcommand{\N}{\mathbb N}%
\newcommand{\Sb}{\mathbb S}%

\newcommand{\Ac}{\mathcal A}%
\newcommand{\T}{\mathbb T}%
\newcommand{\X}{\mathbb X}%
\newcommand{\Hb}{\mathbb H}%

\numberwithin{equation}{section}

\usepackage[colorlinks=true,linkcolor=blue,citecolor=blue]{hyperref}
\makeatletter

\renewcommand\subsubsection{\@secnumfont}{\bfseries}%
\renewcommand\subsubsection{\@startsection{subsubsection}{3}
  \z@{.5\linespacing\@plus.7\linespacing}{-.5em}%
  {\normalfont\bfseries}}

  \makeatother
\makeatletter
\@namedef{subjclassname@2020}{%
  \textup{2020} Mathematics Subject Classification}
\makeatother
\begin{document}

\title[Ces\`aro summability and Talbot effect]{Ces\`aro summability of H\"older functions and Talbot effect on rank one Riemannian symmetric spaces of compact type }

\author[Utsav Dewan]{Utsav Dewan}
\address{Department of Mathematics, Indian Institute of Technology Bombay, Powai, Mumbai-400076, India}
\email{utsav@math.iitb.ac.in}



\begin{abstract}
On rank one Riemannian symmetric spaces of compact type (of dimension $\ge 2$), we first obtain a quantitative characterization of H\"older continuity in terms of Ces\`aro means. In addition to some approximation theoretic applications, we also apply it to study the celebrated physical phenomenon known as `Talbot effect' arising from diffraction theory. More precisely, for almost every fixed time instance, we study the H\"older continuity and the fractal profile of the Schr\"odinger propagation in terms of the decay of the Littlewood-Paley projections of the initial data. In the process, we also obtain oscillatory expansions of zonal spherical functions uniformly near the origin and near the cut locus respectively, which may be of independent interest.
\end{abstract}

\maketitle
\tableofcontents

\section{Introduction}
\label{sec1}
Let $U$ be a connected, simply connected, compact, semi-simple Lie group and $K$ be a closed subgroup with the property that $U^\theta_0 \subset K \subset U^\theta$ where $\theta$ is an involution of $U$, $U^\theta$ denotes the subgroup of $\theta$-fixed points and $U^\theta_0$ its identity component. Let $\mathfrak{u}$ be the Lie algebra of $U$ and $\mathfrak{u}=\mathfrak{k}+\mathfrak{q}$ be the Cartan decomposition corresponding to $\theta$, where $\mathfrak{k}$ is the Lie algebra of $K$. Let $\mathfrak{a} \subset \mathfrak{q}$ be a maximal abelian subspace. Then rank one Riemannian symmetric spaces of compact type are homogeneous spaces $\X=U/K$ such that the dimension of $\mathfrak{a}$ is one. The tangent space at the origin $T_o\X$ (where $o=eK$) is identified with $\mathfrak{q}$, with the underlying Riemannian metric induced by the Killing form. 

In fact, one has a well-known classification of such spaces \cite{Wang}: 
\begin{itemize}
\item the sphere $\Sb^d =SO(d+1)/SO(d)$, $d=1,2,3,\dots$;
\item the real projective space $P^d(\R)=SO(d+1)/O(d)$, $d=2,3,4,\dots$;
\item the complex projective space $P^d(\C)=SU(l+1)/S(U(l) \times U(1))$, $d=4,6,8,\dots$ and $l=d/2$;
\item the quarternionic projective space $P^d(\Hb)=Sp(l+1)/Sp(l) \times Sp(1)$, $d=8,12,16,\dots$ and $l=d/4$;
\item the Cayley projective plane $P^{16}(Cay)$.
\end{itemize}
Here $d$ denotes the real dimension of any of these spaces, $O(d),\:U(d),\:Sp(d)$ denote the orthogonal, unitary and symplectic groups of order $d$ and $S(\cdot)$ denotes the formation of a subgroup of matrices of unit determinant. In this article, we are interested when $d \ge 2$, that is, we exclude the commutative group, the unit circle $\Sb^1$ from our discussion.

The space $\X$ is equipped with the push-forward measure of the normalized Haar measure $du$ of $U$. Let $(\delta,V_\delta)$ be an irreducible unitary representation of $U$ and $V^K_\delta$ be the space of vectors $v \in V_\delta$ fixed under $\delta(K)$. The harmonic analytic aspects of $\X$ are captured by $\what{U}_K$, the collection of equivalence classes of irreducible, unitary representations $\delta$ of $U$ such that $V^K_\delta \ne \{0\}$, as by the Peter-Weyl theory of $U$, we have the decomposition,
\begin{equation} \label{peter-weyl}
L^2(\X) = \bigoplus_{\delta \in \what{U}_K} \mathscr{H}_\delta(\X)\:,
\end{equation}
where
\begin{equation*}
\mathscr{H}_\delta(\X):= \left\{\langle \delta(u){\bf e}, v   \rangle \mid v \in V_\delta\right\}\:,
\end{equation*}
and ${\bf e}$ is the unique unit vector spanning $V^K_\delta$.

By \cite[Theorem 4.1, p. 535]{Helgason}, $\what{U}_K$ can be identified with the highest weight lattice $\Lambda^+(U/K)$, which in rank one is simply given by $\N \cup \{0\}$, the set of non-negative integers. Moreover, the decomposition (\ref{peter-weyl}) is also the eigenspace decomposition of the Laplace-Beltrami operator $\Delta$ on $\X$: for any $f \in \mathscr{H}_n(\X)\:,\:n \in \N \cup \{0\}$,
\begin{equation*}
\Delta f = \lambda^2_n f\:,
\end{equation*}
where $0=\lambda^2_0<\lambda^2_1<\lambda^2_2< \dots$ consist of the discrete $L^2$-spectrum of $\Delta$.

For $f \in L^2(\X)$, we have the following equality in the $L^2$-sense:
\begin{equation*}
f=\sum_{n=0}^\infty proj_nf\:,
\end{equation*}
where for $n \in \N \cup \{0\}$, $proj_nf$ is the orthogonal projection of $f$ onto the subspace $\mathscr{H}_n(\X)$. In other words, the sequence of partial sums 
\begin{equation} \label{partial_sum}
S_Nf:=\sum_{n=0}^N proj_nf\:,
\end{equation}
satisfies
\begin{equation} \label{convergence_partial_sum}
\lim_{N \to \infty} \|S_Nf-f\|_{L^2(\X)}=0\:.
\end{equation}
If however $f \in L^p(\X)$ for $p \in [1,2) \cup (2,\infty)$ or $f \in C(\X)$ (the space of continuous functions on $\X$) for $p=\infty$, the convergence (\ref{convergence_partial_sum}) fails in general in the $L^p$ norm.  This follows from Fefferman's negative solution to the ball multiplier theorem \cite{Fefferman} and from transplantation theorems \cite{Mitjagin}.

This motivates one to look at other summability methods. One of the most prominent example of such a method is that of Ces\`aro means: for $\delta \ge 0$, the Ces\`aro means of index $\delta$ of the sequence of partial sums (\ref{partial_sum}) are defined as,
\begin{equation} \label{Cesaro_mean}
C^\delta_Nf := \frac{1}{A^\delta_N} \sum_{n=0}^N A^\delta_{N-n}\:proj_nf\:,\: N \in \N \cup \{0\}\:,
\end{equation}
where
\begin{equation*}
A^\delta_m= {m+\delta \choose m}\:.
\end{equation*}
Note that if $\delta=0$, $C^\delta_Nf=S_Nf$. In their celebrated work \cite{Bonami}, Bonami and Clerc studied the convergence properties of the Ces\`aro means (\ref{Cesaro_mean}) on $\X$. Among other results, they proved that for $\delta > \frac{d-1}{2}$ and $f \in C(\X)$, 
\begin{equation} \label{cesaro_convergence}
\lim_{N \to \infty} \|C^\delta_Nf-f\|_{L^\infty(\X)}=0\:.
\end{equation} 

An important subclass of $C(\X)$ is $C^\gamma(\X)$, the class of H\"older functions of order $\gamma \in (0,1)$. We recall that $f \in C^\gamma(\X),\:\gamma \in (0,1)$ if
\begin{equation} \label{Holder}
|f(x)-f(y)| \le C d(x,y)^\gamma\:,\: x,y \in \X\:,
\end{equation}
where $C>0$ is some constant and $d(\cdot,\cdot)$ is the inner metric on the Riemannian manifold $\X$. Our first result is a quantitative characterization of $C^\gamma(\X)$, in the spirit of (\ref{cesaro_convergence}):
\begin{theorem} \label{thm_cesaro}
For $d \ge 2$, set \begin{equation} \label{cesaro_delta} \delta_*:=
\begin{cases}
	 d\:,\:&\text{ if } \X=\Sb^d\:,P^d(\R)\:, \\
	 \frac{d+2}{2}\:,\:&\text{ if } \X=P^d(\C)\:,\\
	 \frac{d+4}{2}\:,\:&\text{ if } \X=P^d(\Hb)\:,\\
	 12\:,\:&\text{ if } \X=P^{16}(Cay)\:.
	\end{cases}
\end{equation}
Let $\delta \ge \delta_*$ and $\gamma \in (0,1)$. Then $f \in C^\gamma(\X)$ if and only if $\|C^\delta_Nf-f\|_{L^\infty(\X)} \lesssim N^{-\gamma}\:.$
\end{theorem}  

The necessity of the condition  $\|C^\delta_Nf-f\|_{L^\infty(\X)} \lesssim N^{-\gamma}\:,$ follows by decomposing $\X$ into a small geodesic ball and finitely many dyadic geodesic annuli and then suitably applying the kernel estimates given by (\ref{cesaro_kernel_estimate}) on them. The arguments for the sufficiency part on the other hand, involve the Hopf-Rinow theorem and a genuine analogue of the classical Bernstein inequality for spherical polynomials (Lemma \ref{bernstein}). In fact, we prove a much general statement for the sufficiency of a function to be in $C^\gamma(\X),\:\gamma \in (0,1),$ which is recorded in Corollary \ref{sufficiency_corollary}. Theorem \ref{thm_cesaro} generalizes the corresponding result for $\Sb^d$ \cite[Theorem 2.13]{Huynh} to the broad setting of $\X$. 

Theorem \ref{thm_cesaro} has some important consequences. Firstly, for $f \in C(\X)$ and $x \in \X$, we recall the spherical mean of $f$ at $x$ for $r$ smaller than the diameter of $\X$:
\begin{equation*}
\tau_rf(x):= \int_{S(x,r)} f\:d\sigma_r\:, 
\end{equation*}
where $S(x,r)$ is the geodesic sphere in $\X$, centred at $x$ with radius $r$ and $d\sigma_r$ is the (normalized) induced Riemannian measure on $S(x,r)$. Then the H\"older condition (\ref{Holder}) has an equivalent formulation in terms of $\tau_r$ and the identity operator $I$:
\begin{corollary} \label{corollary1}
Let $f \in C(\X)$ and $\gamma \in (0,1)$. Then $f \in C^\gamma(\X)$ if and only if for $t \to 0+$, we have $\displaystyle\sup_{0<r \le t} \|(\tau_r -I)^\frac{1}{2}f\|_{L^\infty(\X)} \lesssim t^\gamma$\:.
\end{corollary}

Next, we recall another well-known summation method, namely the Riesz means. For $\delta \ge 0$, the Riesz means of index $\delta$ of the sequence of partial sums (\ref{partial_sum}) are defined as,
\begin{equation*}
S^{\delta}_Rf:= \sum_{\lambda_n <R} \left(1-\frac{\lambda_n}{R}\right)^\delta \:proj_nf \:.
\end{equation*}
We next obtain a quantitative characterization of $C^\gamma(\X)$ in terms of asymptotics of their Riesz means:
\begin{corollary} \label{corollary2}
Let $\delta \ge \delta_*,\:\gamma \in (0,1)$ and $f \in C(\X)$. Then $f \in C^\gamma(\X)$ if and only if $\|S^\delta_Rf-f\|_{L^\infty(\X)} \lesssim R^{-\gamma}\:.$
\end{corollary}

\begin{remark} \label{riesz_remark}
In the seminal work \cite{Sogge}, Sogge proved that for $\delta > \frac{d-1}{2}$ and $f \in C(\X)$, 
\begin{equation} \label{riesz_convergence}
\lim_{R \to \infty} \|S^\delta_Rf-f\|_{L^\infty(\X)}=0\:.
\end{equation} 
Thus Corollary \ref{corollary2} can be viewed as a quantitative strengthening of (\ref{riesz_convergence}) for H\"older functions, analogous to what Theorem \ref{thm_cesaro} is compared to (\ref{cesaro_convergence}) for Ces\`aro means.
\end{remark}
  
Next we arrive at Besov spaces. In order to do so, let us first recall that the Littlewood-Paley projections of a function $f$ on $\X$ are given by,
\begin{equation*}
P_Nf:=
\begin{cases}
	 \displaystyle\sum_{n=0}^2  proj_nf\:,\:& N=1\:,\\
	 \displaystyle\sum_{n=N+1}^{2N} proj_nf\:,\:&N\ge 2, \text{ dyadic}\:.
	\end{cases}
\end{equation*}
The Besov spaces $B^\gamma_{p,\infty}(\X)$, for $1 \le p \le \infty$, are defined by the norm in terms of the Littlewood-Paley projections:
\begin{equation} \label{besov_space}
\|f\|_{B^\gamma_{p,\infty}(\X)}:=\sup \left\{N^\gamma\left\|P_Nf\right\|_{L^p(\X)} : N\ge 1, \text{ dyadic}\right\}\:.
\end{equation}
The Besov spaces can be seen as natural generalizations of the Sobolev spaces and hence embedding results of these spaces are of particular interest. 
%(see \cite{Ruzhansky})
As an applicaion of Theorem \ref{thm_cesaro}, we get the following embedding:
\begin{corollary} \label{corollary3}
For $\gamma \in (0,1)$, $B^\gamma_{\infty,\infty}(\X) \subset C^\gamma(\X)$\:.
\end{corollary}

Corollaries \ref{corollary1} and \ref{corollary2} are related to the Peetre $K$-moduli, another crucial notion in approximation theory. But in this article, we will not pursue this direction and refer the interested reader to \cite{Li}. Instead, utilizing the embedding result Corollary \ref{corollary3}, we embark on the study of an interesting physical phenomenon, known as the `Talbot effect'.

The story begins in 1836 when Talbot was studying monochromatic light passing through a diffraction grating \cite{Talbot}. He observed that a sharp focused grating pattern reappears at a certain distance, now known as the Talbot distance. Moreover, at rational multiples of the Talbot distance, the pattern appears to be a finite linear combination of the grating pattern, with a complexity increasing as the denominator of the rational number increases. This physical phenomenon is known as the Talbot effect.
%In \cite{Rayleigh}, Rayleigh also obtained relations between the Talbot distance with the spacing of the grating and the wavelength of the incoming light. 

This study was furthered by Berry and his collaborators in a series of papers \cite{Berry1, Berry2, Berry3, Berry4}. In particular, in \cite{Berry2} Berry and Klein used the Schr\"odinger evolution on $\Sb^1$ to model the Talbot effect and showed that at rational times, the solution is a linear combination of finitely many translates of the initial data with Gauss sums as coefficients. This yields that at rational times, the Schr\"odinger propagation is an $L^r(\Sb^1)$ multiplier for all $r \in [1,\infty]$ with a precise estimate on the operator norm, in terms of the denominator of the rational number (see \cite[Theorem 2.27]{Book}). 

An inquisitive mind is then naturally led to the question of the qualitative properties of the solution at irrational time instances. In this direction, Berry and Klein \cite{Berry2} had already observed that at irrational times, the solutions have a fractal profile. In particular, the Schr\"odinger propagation of a step function at rational times is again a step function but at irrational times it is a continuous but nowhere differentiable function with upper Minkowski dimension $\frac{3}{2}$.  

Let us briefly recall the geometric measure theoretic notion of the upper Minkowski dimension. Let $E$ be a subset of a totally bounded metric space $(\mathcal{X},d)$. For $\varepsilon>0$, let $\mathcal{N}_\varepsilon(E)$ denote the smallest number of metric balls of radii $\le \varepsilon$ which can cover $E$. Then the {\it upper Minkowski dimension} of $E$ is defined to be,
\begin{equation} \label{minkowski_defn}
\overline{\dim}_M(E):= \limsup_{\varepsilon \to 0}\frac{\log \left(\mathcal{N}_\varepsilon(E)\right)}{\log \left(\frac{1}{\varepsilon}\right)}\:.
\end{equation}

Coming back to the literature survey of the Talbot effect, Oskolkov showed that in fact, the Schr\"odinger propagation of a bounded variation function (on $\Sb^1$), is continuous at irrational times \cite{Oskolkov}. Then Kapitanski and Rodnianski furthered the rational-irrational dichotomy by showing that the Schr\"odinger evolution has better regularity properties (measured in the scale of Besov spaces) at irrational times than at rational times \cite{Kapitanski}. This was followed by a host of authors  studying the upper Minkowski dimension of the real and imaginary parts of the solution for almost every time instance for $\Sb^1$ \cite{Rodnianski, CO1, CO2, Chousionis} and also for higher dimensional flat tori $\T^d$ \cite{MathZ}. 

In the case of non-flat compact Riemannian manifolds however, there are only a handful results. In the setting of Zoll manifolds $M$, that is, compact, connected Riemannian manifolds, all of whose geodesics are closed with a common period, Taylor \cite{Taylor} studied the Talbot effect for the Schr\"odinger propagator
\begin{equation} \label{schrodinger}
i\frac{\partial u}{\partial t} +\Delta u=0\:,\:\:	u(\cdot, 0)=f\:, 
\end{equation}
by relating it to a hyperbolic operator at rational times. Then using results from \cite{SSS}, he obtained sharp mapping properties of the Schr\"odinger propagator at rational times between $L^p$-fractional Sobolev spaces: 
\begin{equation*} 
H^{s,p} (M) \to H^{s-(d-1)\left|\frac{1}{2}-\frac{1}{p}\right|,p}(M)\:,
\end{equation*}
where $d$ is the dimension of $M$. The fractal profile analysis at irrational times however has only been studied for $\Sb^d$, recently by Erdo\u{g}an et. al. \cite[Theorem 1.1]{MathZ}. To understand their result we have to go back to \cite{Oskolkov}, that is, bounded variation initial data on $\Sb^1$. For a bounded variation function $f$ on $\Sb^1$, as one has the following decay in its Littlewood-Paley projections:
\begin{equation} \label{Oskolkov_setting}
\|P_Nf\|_{L^p(\Sb^1)} \lesssim N^{-\frac{1}{p}}\:,
\end{equation}
the authors naturally generalized this to the following requirement on functions on $\Sb^d$:
\begin{equation*} 
\|P_Nf\|_{L^p(\Sb^d)} \lesssim N^{-\frac{d}{2}-s}\:,
\end{equation*}
for $p \ge 1$ and some extra smoothing parameter $s\ge 0$ and then studied the fractal profile of the propagator at almost every time instance.

The most well-known examples of Zoll manifolds are given by the class of rank one Riemannian symmetric spaces of compact type. Thus to complement Taylor's result by studying the rational-irrational dichotomy, it seems natural to aim for a generalization of \cite[Theorem 1.1]{MathZ} to all rank one Riemannian symmetric spaces of compact type. In fact, we obtain the following stronger result:
\begin{theorem} \label{thm_talbot}
Let $\X$ be a $d$-dimensional rank one Riemannian symmetric space of compact type $(d \ge 2)$ and $f$ be a real-valued function on $\X$. Set
 \begin{equation*}
s  \begin{cases}
       \ge \frac{d}{p}-\frac{d+1}{2}\:\:&\text{ if } 1 \le p < \frac{2d}{d+1}\:,\\
	   > 0\:\:&\text{ if } p = \frac{2d}{d+1}\:,\\
	 \ge 0\:\:&\text{ if } \frac{2d}{d+1} < p \le \infty\:.
	\end{cases}
\end{equation*}
Assume that $\|P_Nf\|_{L^p(\X)} \lesssim N^{-\left(\frac{d}{2}+s\right)}$, for $N\ge 1$, dyadic. Then 
\begin{enumerate}
\item \label{talbot_part1} for almost all $t$, the solution $u(\cdot,t)$ to (\ref{schrodinger}) in $C^{\gamma'}(\X)$ 
for $\gamma'=\min\{\gamma,1\}$ where 
 \begin{equation*}
\gamma =\begin{cases}
      \left(s+\frac{d+1}{2}-\frac{d}{p}\right)\:\:&\text{ if } 1 \le p < \frac{2d}{d+1}\:,\\
	  s-\:\:&\text{ if } p= \frac{2d}{d+1}\:,\\
	 	s\:\:&\text{ if } \frac{2d}{d+1} < p \le \infty\:;
	\end{cases}
\end{equation*}
\item \label{talbot_part2} and hence, for almost all $t$,
\begin{equation*}
\max\left\{\overline{\dim}_M\left(Re\left(u(\cdot,t)\right)\right)\:,\:\overline{\dim}_M\left(Im\left(u(\cdot,t)\right)\right)\right\} \le (d+1)-\min\{\gamma'',1\}\:,
\end{equation*}
where
\begin{equation*}
\gamma'' =\begin{cases}
      \left(s+\frac{d+1}{2}-\frac{d}{p}\right)\:\:&\text{ if } 1 \le p < \frac{2d}{d+1}\:,\\
	   	s\:\:&\text{ if } \frac{2d}{d+1} \le p \le \infty\:.
	\end{cases}
\end{equation*}
\end{enumerate}
\end{theorem}
\begin{remark} \label{smoothing_effect}
The part (\ref{talbot_part1}) of Theorem \ref{thm_talbot} means that if the initial data $f \in B^{\frac{d}{2}+s}_{p,\infty}(\X),$ $p \in [1,\infty]$ then for almost every time instance $t$, the propagation $u(\cdot,t) \in C^{\gamma'}(\X)$. In fact, the proof of Theorem \ref{thm_talbot} will show that for almost every time instance $t$, the propagation $u(\cdot,t) \in B^{\gamma}_{\infty,\infty}(\X)$.
\end{remark}

The corresponding region of admissible pairs for the $L^p$ scale of Littlewood-Paley projections and the extra regularity parameter $s$ obtained in Theorem \ref{thm_talbot}, is given by the shaded region in Figure \ref{figure}.
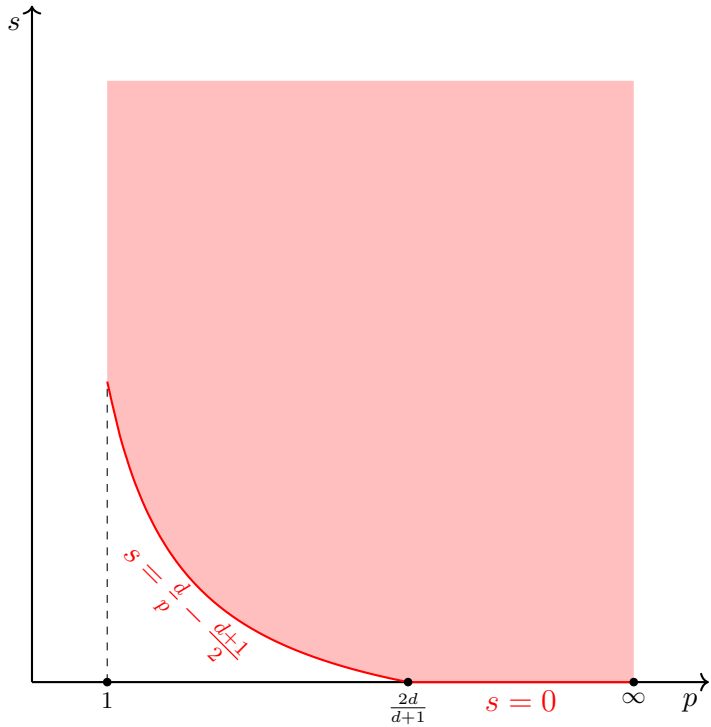
\begin{figure}[h]
\begin{center}
\begin{tikzpicture}
	%\draw[help lines] (0,0) grid (9,9);
	      \draw [->, thick] (0,0) -- (9,0) node [below left] {\small{$p$}};
	\draw [->, thick] (0,0) -- (0,9)  node [below left] {\small$s$};
\fill[pink] (8,0)--(5,0) -- plot[domain=5:1, smooth] (\x,{(5/\x)-1}) -- (1,4) -- (1,8) -- (8,8)-- cycle;	
	\draw[red][thick] plot[domain=5:1, smooth] (\x,{(5/\x)-1});
	\draw[red][thick] (5,0)--(8,0);
	\pgftext[at={\pgfpoint{2cm}{1cm}}, rotate=-45]{\textcolor{red}{$s=\frac{d}{p}-\frac{d+1}{2}$}};
	\draw [fill] (5,0) circle [radius=.05] node [below] {\scriptsize$\frac{2d}{d+1}$};
	\draw [fill] (1,0) circle [radius=.05] node [below] {\scriptsize$1$};
	\draw [dashed] (1,0) -- (1,4);
     \pgftext[at={\pgfpoint{6.5cm}{-0.25cm}}, rotate=0]{\textcolor{red}{$s=0$}};	
     \draw [fill] (8,0) circle [radius=.05] node [below] {\scriptsize$\infty$};
   	\end{tikzpicture}
\end{center}
\caption{$L^p$ scale of Littlewood-Paley projections and the extra regularity parameter $s$}
\label{figure}
\end{figure}

As an immediate consequence, we obtain the following result in the spirit of bounded variation initial data on $\Sb^1$ (\ref{Oskolkov_setting}):
\begin{corollary} \label{talbot_corollary}
Let $\X$ be a $d$-dimensional rank one Riemannian symmetric space of compact type $(d \ge 2)$, $f:\X \to \R,\:p=\frac{2d}{d+1}$ and assume that  $\|P_Nf\|_{L^p(\X)} \lesssim N^{-\frac{d}{p}}$, for $N\ge 1$, dyadic. Then for almost all $t$, the solution $u(\cdot,t)$ to (\ref{schrodinger}) is in $C^{\frac{1}{2}-}\left(\X\right)$ and hence 
\begin{equation*}
\max\left\{\overline{\dim}_M\left(Re\left(u(\cdot,t)\right)\right)\:,\:\overline{\dim}_M\left(Im\left(u(\cdot,t)\right)\right)\right\} \le d+\frac{1}{2}\:.
\end{equation*}
\end{corollary}

The proof of Theorem \ref{thm_talbot} follows the general scheme of the arguments employed by Erdo\u{g}an et. al. \cite{MathZ} in the special case of $\Sb^d$, namely showing that for almost every fixed time instance, the Schr\"odinger propagation is in the Besov space $B^{\gamma}_{\infty,\infty}$, hence by the Besov embedding (Corollary \ref{corollary3}), it is in the H\"older space $C^\gamma$ and from there the upper bound on the fractal dimension would follow. The two main ingredients in carrying out the above scheme are the following:

\medskip

(1) Uniform oscillatory expansions of the zonal spherical functions:

\begin{itemize}
\item Proving the inclusion in the Besov space $B^{\gamma}_{\infty,\infty}$ follows by writing Littlewood-Paley projections of the (time fixed) Schr\"odinger propagation as a right convolution with a $K$-biinvariant kernel and then by Young's inequality, it boils down to getting pointwise estimates for the kernel. This however becomes tricky as it requires to estimate the joint oscillation afforded by the Sch\"odinger multiplier and the zonal spherical functions. In the case of $\Sb^d$ this is quite easily overcome by looking at a two-term Bessel expansion of the Gegenbauer polynomials uniformly in $\left[0,\frac{\pi}{2}\right]$ \cite[Theorem 1.2]{Baratella} and then exploiting the symmetry of the Gegenbauer polynomials in $\left[\frac{\pi}{2},\pi\right]$.

\medskip

\item In the general setting of rank one Riemannian symmetric space of compact type however, the zonal spherical functions do not enjoy such symmetries. In fact, due to the geometric technicality of concentration of higher dimensional cut loci, compared to the singleton antipodal points in $\Sb^d$ (compare the values of $M_1$ in Table \ref{table:1}), the behaviour of the zonal spherical functions are quite different near the cut locus, compared to near the origin (see Remark \ref{epxansion_symmetry_sphere}). This is made explicit in Theorem \ref{thm_oscillatory_zn}, where we obtain two oscillatory expansions of zonal spherical functions, uniformly near the origin and near the cut locus respectively, which may be of independent interest. The proof of Theorem \ref{thm_oscillatory_zn} requires to utilize a general $m$-term expansion obtained by Frenzen and Wong \cite{Frenzen} depending on the asymmetry of the Jacobi parameters $\alpha,\beta$ (see Table \ref{table:2}).   
\end{itemize}

\medskip

(2) Upper Minkowski dimension of the graph of a H\"older function:

\begin{itemize}
\item After establishing the H\"older continuity of the Schr\"odinger propagation for almost every fixed time instance, the conclusion about the fractal dimension follows from a geometric measure theory result of the form: the graph of any real-valued $C^\gamma$ function, $\gamma \in (0,1),$ has upper Minkowski dimension  $\le (d+1)-\gamma$, where $d$ is the topological dimension of the underlying space. In the case of a bounded domain in $\R^d$, this follows from a standard counting argument involving boxes (see Lemma \ref{gmt_lemma3}). In the case of $\Sb^d$, the situation is quite similar due to its canonical embedding into $\R^{d+1}$ (see \cite{Huynh}). 

\medskip

\item In the case of the other rank one Riemannian symmetric spaces of compact type, that is, projective spaces however, there is no such simple embedding available due to complicated topological obstructions such as Hopf fibrations. This prompts us to take a rather abstract approach by utilizing the finite stability of the upper Minkowski dimension and the local bi-Lipschitz structure of compact, connected Riemannian manifolds. These ideas lead us to study the problem in a vastly general setting of compact metric spaces that are endowed with a finite bi-Lipschitz structure. Then by gluing the local result (Lemma \ref{gmt_lemma3}), we obtain Theorem \ref{thm_gmt} and consequently, the desired result for compact, connected Riemannian manifolds (Corollary \ref{graph_corollary}).
\end{itemize}

This article is organized as follows. In Section \ref{sec2}, we recall the required preliminaries on rank one Riemannian symmetric spaces of compact type and harmonic analysis thereon and also fix our notations. In Section \ref{sec3}, we prove a genuine analogue of the classical Bernstein inequality for spherical polynomials (Lemma \ref{bernstein}), Theorem \ref{thm_cesaro} and a useful general condition for sufficiency of H\"older continuity (Corollary \ref{sufficiency_corollary}). In Section \ref{sec4}, Corollaries \ref{corollary1}, \ref{corollary2} and \ref{corollary3} are proved. In Section \ref{sec5}, we estimate the upper Minkowski dimension of graphs of real-valued H\"older functions on general compact metric spaces and compact, connected Riemannian manifolds: Theorem \ref{thm_gmt} and Corollary \ref{graph_corollary}. In Section \ref{sec6}, we obtain oscillatory expansions of zonal spherical functions of degree $\ge 1$ uniformly near the origin and near the cut locus respectively: Theorem \ref{thm_oscillatory_zn}. In Section \ref{sec7}, we prove our results on the Talbot effect: Theorem \ref{thm_talbot} and Corollary \ref{talbot_corollary}. Finally in Section \ref{sec8}, we conclude by making some remarks and posing some new problems. 

\section{Preliminaries}
\label{sec2}
\subsection{Some notations} \label{subsec2.1}
Throughout this article $c,C,\dots$ will be used to denote positive constants whose values may change on each occurrence. $\N$ will denote the set of positive integers. Two positive functions $f_1$ and $f_2$ will satisfy,
\begin{itemize}
\item  $f_1 \lesssim f_2$ if there exists $C\ge 1$ such that $f_1 \le Cf_2$;
\item  $f_1 \gtrsim f_2$ if there exists $C\ge 1$ such that $Cf_1 \ge f_2$;
\item $f_1 \asymp f_2$ if there exists $C\ge 1$ such that $\frac{1}{C}f_1 \le f_2 \le Cf_1$.
\end{itemize}
The notation $\lesssim_\varepsilon$ will also be used to denote the dependence of the parameter $\varepsilon$. In addition, if $f_1$ is complex-valued, we will write,
\begin{itemize}
\item $f_1=\mathcal{O}(f_2)$ to denote that $|f_1| \lesssim f_2$\:.
\end{itemize} 
We will denote $s+$ (or $s-$ respectively) to denote $s+\varepsilon$ (or $s-\varepsilon$ respectively) for all $\varepsilon>0$, with the implicit constants possibly depending on $\varepsilon$. For $s \in (0,1),$ we use the following notations:
\begin{equation*}
C^{s-}\left(\X\right):=\displaystyle\bigcap_{r<s} C^{r}\left(\X\right)\:,\text{ and } \:\: B^{s-}_{\infty,\infty}\left(\X\right):=\displaystyle\bigcap_{r<s} B^{r}_{\infty,\infty}\left(\X\right)\:.
\end{equation*}
Finally, for $x \in \R$, $\langle x \rangle:=\left(1+|x|^2\right)^{\frac{1}{2}}$.

\subsection{Harmonic analysis on rank one Riemannian symmetric spaces of compact type} \label{subsec2.2}

In this subsection, we briefly recall some preliminaries on rank one Riemannian symmetric spaces of compact type and harmonic analysis thereon. The relevant information can be found in \cite{Bonami,HelgasonDiff, Helgason,CT}.

Without loss of generality, the Riemannian measure $\mu$ and the metric on $\X$ can be renormalized so that the total measure of $\X$ is $1$ and the diameter of $\X$ is $\pi$. For a fixed origin $o \in \X$, the points with distance from $o$ equal to the diameter of $\X$ are the antipodal points of $o$. The collection of all antipodal points of $o$ is termed as the {\em antipodal manifold}. Denoting the dimension of the antipodal manifold as $M_1$ and the dimension of $\X$ as $d$, we introduce a non-negative integer $M_2$ so that $M_1+M_2+1=d$. Furthermore, if $\Ac(\theta)$ denotes the Riemannian measure of a geodesic sphere in $\X$, centered at $o$ with radius $\theta$, for $\theta \in [0,\pi]$, then it has the explicit form,
\begin{equation} \label{density}
\Ac(\theta)= C {\left(\sin \frac{\theta}{2}\right)}^{M_1} {(\sin \theta)}^{M_2}\:,
\end{equation}
where the constant $C>0$ is such that $\int_0^\pi \Ac(\theta) d\theta=1$. The specific values of the parameters $M_1,M_2$ are given in Table \ref{table:1}.
 
\begingroup

\setlength{\tabcolsep}{10pt} % Default value: 6pt
\renewcommand{\arraystretch}{1.5} % Default value: 1
\begin{table}[H]
\centering
\begin{tabular}{ ||p{1.7cm}|| p{1cm} | p{1cm} || p{1.7cm}|| p{1cm} | p{1cm}| }
 %\hline
 %\multicolumn{4}{|c|}{} \\
 \hline
 $\X$ & $M_1$ & $M_2$ & $\X$ & $M_1$ & $M_2$ \\
 \hline 
$\Sb^d$   & $0$   & $d-1$ & $P^d(\Hb)$ & $d-4$ & \:\:$3$  \\
$P^d(\R)$ & $d-1$   & \:\:$0$ & $P^{16}(Cay)$ & \:\:$8$ & \:\:$7$ \\
$P^d(\C)$ & $d-2$  & \:\:$1$ & & &   \\
\hline
\end{tabular}
\medskip
\caption{}
\label{table:1}
\end{table}
\endgroup

We recall from the introduction that $\mathscr{H}_n(\X)$ are the finite dimensional eigenspaces of $\Delta$ with eigenvalues 
\begin{equation} \label{rank1_eigenvalue_form} 
\lambda^2_n =sn(sn+\alpha+\beta+1)\:,\:\:n \in \N \cup \{0\}\:,
\end{equation}  
with 
\begin{equation} \label{value_of_s}
s=\begin{cases}
	 1  &\text{ if }  \X = \Sb^d,\:P^d(\C),\:P^d(\Hb),\:P^{16}(Cay), \\
	2 &\text{ if }  \X = \:P^d(\R) \:.
	\end{cases}	
\end{equation}
The parameters $\alpha, \beta$ and the eigenvalues $\lambda^2_n$ are given in Table \ref{table:2}.

\begingroup

\setlength{\tabcolsep}{10pt} % Default value: 6pt
\renewcommand{\arraystretch}{1.5} % Default value: 1
\begin{table}[h!]
\centering
\begin{tabular}{ ||p{1.7cm}|| p{1cm}| p{1cm}| p{2.3cm}|  }
 %\hline
 %\multicolumn{4}{|c|}{} \\
 \hline
 $\X$ & $\alpha$ & $\beta$ & $\lambda^2_n$\\
 \hline 
$\Sb^d$   & $\frac{d-2}{2}$   & $\frac{d-2}{2}$ & $n(n+d-1)$ \\
$P^d(\R)$ & $\frac{d-2}{2}$   & $\frac{d-2}{2}$  & $2n(2n+d-1)$ \\
$P^d(\C)$ & $\frac{d-2}{2}$  & \:\:$0$  & $n\left(n+\frac{d}{2}\right)$ \\
$P^d(\Hb)$ & $\frac{d-2}{2}$ & \:\:$1$ & $n\left(n+1+\frac{d}{2}\right)$   \\
$P^{16}(Cay)$ & \:\:$7$ & \:\:$3$  & $n(n+11)$ \\
\hline
\end{tabular}
\medskip
\caption{}
\label{table:2}
\end{table}
\endgroup
\noindent Let $d_n :=dim\left(\mathscr{H}_n(\X)\right)$ and $\{Y_{n,j}\}_{j=1}^{d_n}$ be an orthonormal basis of $H_n(\X)$. Then $d_n$ satisfies, 
\begin{equation} \label{dimension_growth}
d_n \asymp \left(1+n\right)^{d-1}\:.
\end{equation}

\begin{definition} \label{defn_spherical_poly}
A {\it spherical polynomial in $\X$ of degree at most $m$} $\in \N \cup \{0\}$, is an element in the linear span of $\{Y_{n,j} : 1 \le j \le d_n, 1 \le n \le m\}$\:.
\end{definition}

We have the following Bernstein type estimate for directional derivatives of spherical polynomials:

\begin{lemma} \cite[Lemma 3.2]{Platonov} \label{primitive_bernstein}
Let $x \in \X$ and $v$ be a unit tangent vector at $x$. Let $\gamma_{x,v}$ be the geodesic starting at $x$ with initial velocity $v$. Then for $t \in \R$ and $\Phi$, any spherical polynomial of degree at most $m$ on $\X$, we have
\begin{equation*}
\left|\frac{d}{ds}\Phi\left(\gamma_{x,v}(s)\right)\large{|}_{s=t}\right| \le m \|\Phi\|_{L^\infty(\X)}\:.
\end{equation*}
\end{lemma}
The left action of $U$ on $\X$ by isometries is transitive and hence given any $x \in \X$, there exists $u_x \in U$ such that $u_x \cdot o=x$. Moreover as $K$ is the stabilizer of $o$ in $U$, a function $f$ on $\X$ is naturally identified as a function $f$ on $U$ which is invariant under the right $K$-translations. Also as the Riemannian volume measure $\mu$ is given by the push-forward of the Haar measure on $U$, the convolution on $U$ descends to a convolution on $\X$:
\begin{equation}\label{convolution}
\left(f*g\right)(x):= \int_{\X} f(u_y\cdot o)g(u^{-1}_y u_x \cdot o)\:d\mu(y)\:,\:\:x\in \X\:,
\end{equation}
whenever the integral exists.

By the two-point homogenity, $K$ fixes $o$ and acts transitively on the set of points at a given distance from $o$ (a speciality of rank one). Thus $K$-biinvariant functions on $U$ can be identified with functions on $\X$ {\em radial around $o$}, that is functions $f$ on $\X$ such that $f(x)$ only depends on the geodesic distance of $x$ from $o$. Henceforth, we will refer to such functions simply as radial functions.

In the definition of convolution (\ref{convolution}), if the function $g$ is radial, then we note that
\begin{equation*}
g(u^{-1}_y u_x \cdot o)=g\left(d\left(o,u^{-1}_y u_x \cdot o\right)\right)\:.
\end{equation*}
Then by left $U$-invariance of the metric, we further have
\begin{equation*}
g\left(d\left(o,u^{-1}_y u_x \cdot o\right)\right)=g\left(d\left(u_y\cdot o, u_x \cdot o\right)\right)=g(d(y,x))\:.
\end{equation*} 
Hence, the convolution (\ref{convolution}) takes the special form,
\begin{equation}\label{convolution_radial}
\left(f*g\right)(x):= \int_{\X} f(y)g(d(x,y))\:d\mu(y)\:,\:\:x\in \X\:,
\end{equation}
whenever the integral exists.

For $n \in \N \cup \{0\}$, $proj_nf$, the projection of functions $f$ on $\X$ onto $\mathscr{H}_n(\X)$, is given as a convolution:
\begin{equation} \label{projection_as_convolution}
proj_nf=f*Z_n\:,
\end{equation}
where $Z_n$ is the {\em zonal spherical function of degree $n$}. $Z_n$ is a radial function on $\X$ given by,
\begin{equation*}
Z_n(d(x,y)):=\sum_{j=1}^{d_n} Y_{n,j}(x)\:\overline{Y_{n,j}(y)}\:.
\end{equation*}
These functions are more explicitly understood in terms of Jacobi polynomials:
\begin{equation} \label{in_terms_jacobi}
Z_n(d(x,y)):= d_n \frac{P^{(\alpha,\beta)}_n{\left(\cos\left(\frac{d(x,y)}{s}\right)\right)}}{P^{(\alpha,\beta)}_n(1)}\:,\:\:x,y \in \X\:,
\end{equation}
where $s$ is as in (\ref{value_of_s}).  The Jacobi polynomials satisfy the fundamental symmetry property:
\begin{equation} \label{jacobi_symmetry}
P^{(\alpha,\beta)}_n(-x)=(-1)^n\:P^{(\beta,\alpha)}_n(x)\:,\: x\in [-1,1]\:.
\end{equation}
As $\alpha=\frac{d-2}{2}$ always in our setting, we also have for $n \in \N$,
\begin{equation} \label{jacobi_at_1}
P^{(\alpha,\beta)}_n(1) = {n+\alpha \choose n} \asymp n^\alpha=n^{\frac{d-2}{2}}\:.
\end{equation}
The following expansion of the Jacobi polynomials will be relevant for us:
\begin{lemma} \cite{Frenzen}\label{frenzen}
For $\alpha>-\frac{1}{2},\:\alpha-\beta>-2m$ and $\alpha +\beta \ge -1$, we have
\begin{eqnarray*}
P^{(\alpha,\beta)}_n(\cos \theta)&=&\frac{\Gamma(n+\alpha+1)}{n!}\left(\sin\frac{\theta}{2}\right)^{-\alpha}\left(\cos\frac{\theta}{2}\right)^{-\beta}\left(\frac{\theta}{\sin\theta}\right)^{\frac{1}{2}}\\
&& \times\left[\sum_{l=0}^{m-1} A_l(\theta)\frac{J_{\alpha + l}(N \theta)}{N^{\alpha+l}}+\theta^{\alpha + m} \mathcal{O}\left(N^{-m}\right)\right]\:,
\end{eqnarray*}
where 
\begin{equation*}
\left|A_l(\theta)\right| \lesssim \theta^l\:,\:N=n+\frac{\alpha + \beta +1}{2}\:,\:\theta \in \left[0,\frac{\pi}{2}\right]\:.
\end{equation*}
\end{lemma}
For more details about Jacobi polynomials see \cite{Szego}.

As a consequence of (\ref{projection_as_convolution}), the Ces\`aro means are also given as right convolutions with a radial kernel:
\begin{equation*}
C^\delta_Nf=f*K^\delta_N\:,
\end{equation*}
that is,
\begin{equation*}
C^\delta_Nf(x)=\int_{\X} f(y)\:
K^\delta_N(d(x,y))\:d\mu(y)\:,
\end{equation*}
where 
\begin{equation*}
K^\delta_N(d(x,y))=\frac{1}{A^\delta_N} \sum_{n=0}^N A^\delta_{N-n}\:Z_n(d(x,y))\:.
\end{equation*}
We have,
\begin{equation} \label{integral=1}
\int_{\X} K^\delta_N(d(x,y))\:d\mu(y)=1\:,\:\: x \in \X\:.
\end{equation}
We also have the pointwise estimate of the Ces\`aro kernel on $\X=\Sb^d,\:P^d(\C),\:P^d(\Hb)$ and $P^{16}(Cay)$, for $\delta \ge \delta^*$ (defined in \ref{cesaro_delta}) by \cite[Eqn. (B.1.13), Lemma B.1.2]{Dai} and Table \ref{table:2}:
\begin{equation} \label{cesaro_kernel_estimate}
0 \le K^\delta_N(d(x,y)) \lesssim N^{-1}\left(1-\cos(d(x,y)) +N^{-2}\right)^{-\frac{d+1}{2}}\:, \:\:x,y \in \X\:.
\end{equation}

\section{Ces\`aro summability of H\"older functions}
\label{sec3}
In this section, we prove Theorem \ref{thm_cesaro}. But first, we obtain a genuine analogue of the classical Bernstein inequality for spherical polynomials (see Definition \ref{defn_spherical_poly}) in our present setting:
\begin{lemma} \label{bernstein}
Let $\Phi$ be a spherical polynomial of degree at most $m$ on $\X$. Then we have the following gradient estimate:
\begin{equation*}
\left\|\nabla \Phi\right\|_{L^\infty(\X)} \lesssim m \left\|\Phi\right\|_{L^\infty(\X)}\:,
\end{equation*} 
where the implicit constant depends only on the intrinsic geometry of $\X$.
\end{lemma}
\begin{proof}
Let $x \in \X$. Then in normal coordinates based at $x$, we have the following local form of the gradient:
\begin{equation*} 
\nabla \Phi(x)=\sum_{i=1}^d \left(\sum_{j=1}^d g^{ij}(x)\:\frac{\partial \Phi}{\partial x_j}(x)\right)\frac{\partial}{\partial x_i}\:,
\end{equation*}
where $g^{ij}$ are the components of the inverse metric tensor $g^{-1}$, $\frac{\partial \Phi}{\partial x_j}$ are the directional derivatives in the normal coordinates and $\frac{\partial}{\partial x_i}$ are the coordinate vector fields forming the corresponding basis for $T_x\X$. By compactness of $\X$, there exists $C>0$ such that for all $x \in \X, 1 \le i,j \le d, \:\left|g^{ij}(x)\right| \le C.$ Hence, it suffices to obtain inequalities of the form,
\begin{equation*}
\left|\frac{\partial \Phi}{\partial x_j}(x)\right| \le m \|\Phi\|_{L^\infty(\X)}\:, 
\end{equation*}
which is provided by Lemma \ref{primitive_bernstein}. This completes the proof.
\end{proof}

\begin{proof}[Proof of Theorem \ref{thm_cesaro}]
We recall that the real projective space $P^d(\R)$ can be obtained from $\Sb^d$ by identifying the antipodal points:
\begin{eqnarray*}
\pi : &&\Sb^d\:\: \to \:\: P^d(\R) \\
&& \pm x \:\:\mapsto \:\: [x]\:,
\end{eqnarray*}
with the projection map $\pi$ being a local isometry. Then the distances on $P^d(\R)$ and $\Sb^d$ are related as follows:
\begin{equation*}
d_{P^d(\R)}\left(\pi(x),\pi(y)\right)= \min \left\{d_{\Sb^d}(x,y)\:,\:d_{\Sb^d}(x,-y)\right\}\:,\:\:x,y \in \Sb^d\:.
\end{equation*}
Thus there is a one-to-one correspondence between the H\"older spaces $C^\gamma\left(P^d(\R)\right)$  and $C^\gamma\left(\Sb^d\right)_e$, the collection of even functions in $C^\gamma\left(\Sb^d\right),\:\gamma \in (0,1)$. So, the validity of Theorem \ref{thm_cesaro} on $P^d(\R)$ follows from that of $\Sb^d$ \cite[Theorem 2.13]{Huynh}. Hence, we will prove Theorem \ref{thm_cesaro} for the rest of the cases, that is, $\X=P^d(\C),\:P^d(\Hb)$ and $P^{16}(Cay)$.

For $\delta \ge \delta_*$ and $\gamma \in (0,1)$, we first assume that $f \in C^\gamma(\X)$. Let $x \in \X$. Then by using the property (\ref{integral=1}) and non-negativity (\ref{cesaro_kernel_estimate}) , we get
\begin{eqnarray} \label{cesaro_pf_eq1}
\left|C^\delta_Nf(x)-f(x)\right| = \left|f*K^\delta_N(x) - f(x)\right| &=& \left|\int_{\X} \left(f(y)-f(x)\right)\:
K^\delta_N(d(x,y))\:d\mu(y)\right|\nonumber\\
&\le & \int_{\X}  |f(y)-f(x)|\:K^\delta_N(d(x,y))\:d\mu(y)\:.
\end{eqnarray}
At this juncture, a brief glance at the pointwise estimate (\ref{cesaro_kernel_estimate}):
\begin{equation*} 
 K^\delta_N(d(x,y)) \lesssim N^{-1}\left(1-\cos(d(x,y)) +N^{-2}\right)^{-\frac{d+1}{2}}\:, 
\end{equation*}
motivates us to decompose the last integral into two parts: in a small geodesic ball centred at $x$ and in its complement.

We first do the estimate on the geodesic ball $B\left(x,N^{-1}\right)$. In this case,
\begin{equation*}
 K^\delta_N(d(x,y)) \lesssim N^d\:.
\end{equation*}
Thus,
\begin{eqnarray*}
\int_{B\left(x,N^{-1}\right)}  |f(y)-f(x)|\:K^\delta_N(d(x,y))\:d\mu(y) &\lesssim & N^d \int_{B\left(x,N^{-1}\right)} d(x,y)^\gamma\:d\mu(y) \\
& \le & N^{d-\gamma} \mu\left(B\left(x,N^{-1}\right)\right)\:.
\end{eqnarray*}
Then by the left invariance of the metric and the formula for density (\ref{density}) yields $\mu\left(B\left(x,N^{-1}\right)\right) \lesssim N^{-d}$ and hence,
\begin{equation} \label{cesaro_pf_eq2}
\int_{B\left(x,N^{-1}\right)}  |f(y)-f(x)|\:K^\delta_N(d(x,y))\:d\mu(y) \lesssim N^{-\gamma}\:.
\end{equation}

Next, we shift our attention to $\X \setminus B\left(x,N^{-1}\right)$. For $0 \le k <1+\log_2\left(N^2\right)$, we consider dyadic geodesic annuli of the form,
\begin{equation*}
A_k(x):=\left\{y \in \X : 2^{-\frac{k}{2}} \lesssim d(x,y)\lesssim 2^{-\frac{k-1}{2}}\right\}\:,
\end{equation*}
and note that
\begin{equation*}
\X \setminus B\left(x,N^{-1}\right) \:\:\subset \bigcup_{0\le k < 1+\log_2\left(N^2\right)} A_k(x)\:.
\end{equation*}

For $y \in A_k(x)$, recalling that diameter of $\X=\pi$, we see that
\begin{equation*}
1-\cos (d(x,y)) \gtrsim d(x,y)^2 \gtrsim 2^{-k} \:,
\end{equation*}
and hence,
\begin{equation*}
 K^\delta_N(d(x,y)) \lesssim N^{-1}2^{\frac{k(d+1)}{2}}\:.
\end{equation*}
Thus,
\begin{eqnarray*}
&&\int_{\X \setminus B\left(x,N^{-1}\right)}  |f(y)-f(x)|\:K^\delta_N(d(x,y))\:d\mu(y) \\
&\le & \sum_{k=0}^{2\log_2(N)+1} \int_{A_k(x)}  |f(y)-f(x)|\:K^\delta_N(d(x,y))\:d\mu(y) \\
& \lesssim & N^{-1} \sum_{k=0}^{2\log_2(N)+1} 2^{\frac{k(d+1)}{2}} \int_{A_k(x)} d(x,y)^\gamma\:d\mu(y) \\
& \lesssim & N^{-1} \sum_{k=0}^{2\log_2(N)+1} 2^{\frac{k(d+1)}{2}} 2^{-\frac{(k-1)\gamma}{2}}\mu\left(A_k(x)\right)\:. 
\end{eqnarray*}
Again by the left invariance of the metric and the formula for density (\ref{density}) yields 
\begin{equation*}
\mu\left(A_k(x)\right) \lesssim \mu\left(B\left(x,2^{-\frac{k-1}{2}}\right)\right) \lesssim 2^{-\frac{(k-1)d}{2}}\:,
\end{equation*}
which upon plugging in the last sum yields,
\begin{eqnarray} \label{cesaro_pf_eq3}
\int_{\X \setminus B\left(x,N^{-1}\right)}  |f(y)-f(x)|\:K^\delta_N(d(x,y))\:d\mu(y) & \lesssim & N^{-1} \sum_{2^k \lesssim N^2} 2^{\frac{k}{2}(1-\gamma)} \nonumber\\
& \lesssim & N^{-1} N^{1-\gamma} \nonumber\\
&=& N^{-\gamma}\:.
\end{eqnarray}
Plugging (\ref{cesaro_pf_eq2}) and (\ref{cesaro_pf_eq3}) in (\ref{cesaro_pf_eq1}) yields,
\begin{equation*}
\left|C^\delta_Nf(x)-f(x)\right| \lesssim N^{-\gamma}\:.
\end{equation*}
Since $x$ was arbitrary, we obtain
\begin{equation*}
\left\|C^\delta_Nf-f\right\|_{L^{\infty}(X)} \lesssim N^{-\gamma}\:,
\end{equation*}
thus completing the proof of one implication.

\medskip

Conversely, let $\|f-C^\delta_Nf\|_{L^\infty(\X)} \lesssim N^{-\gamma}\:.$ We note that $C^\delta_N$ is a spherical polynomial of degree at most $N$ (see Definition \ref{defn_spherical_poly}). We will prove the more general statement: if there exists a sequence of spherical polynomials $\mathcal{P}_{2^k}$ of degree $\le 2^k, k \in \N \cup \{0\}$, with $\|f-\mathcal{P}_{2^k}\|_{L^\infty(\X)} \lesssim 2^{-k\gamma}$ then $f \in C^\gamma(\X)$.

By assumption, it follows that
\begin{equation} \label{cesaro_pf_eq4}
f=\mathcal{P}_1 \:+\:\sum_{k=1}^\infty \left(\mathcal{P}_{2^k}-\mathcal{P}_{2^{k-1}}\right)= \mathcal{P}_1 \:+\:\sum_{k=1}^\infty \mathcal{Q}_k\:,
\end{equation} 
where $\mathcal{Q}_k:=\mathcal{P}_{2^k}-\mathcal{P}_{2^{k-1}}\:.$ Then by the triangle inequality,
\begin{equation} \label{cesaro_pf_eq5}
\left\|\mathcal{Q}_k\right\|_{L^\infty(\X)} \le \|f-\mathcal{P}_{2^k}\|_{L^\infty(\X)} + \|f-\mathcal{P}_{2^{k-1}}\|_{L^\infty(\X)} \lesssim 2^{-k\gamma}\:.
\end{equation}
For $x \ne y$, there exists a unique $k_0 \in \N$ such that
\begin{equation*}
2^{-k_0} \pi < d(x,y) \le 2^{-(k_0-1)} \pi\:.
\end{equation*}
Thus for $1 \le k \le k_0-1$, $d(x,y)2^k \pi^{-1} \le 1$ and for $k \ge k_0$, $d(x,y)2^k \pi^{-1} > 1$. Keeping this in mind, for $1 \le k \le k_0-1$, we aim to obtain a sharper estimate for $|\mathcal{Q}_k(x) - \mathcal{Q}_k(y)|$ compared to (\ref{cesaro_pf_eq5}). By the Hopf-Rinow theorem there exists a length-minimizing geodesic segment, say $\zeta$ connecting $x$ to $y$. Then by the mean-value inequality along $\zeta$, we have
\begin{equation*}
|\mathcal{Q}_k(x) - \mathcal{Q}_k(y)| \le \left\|\nabla \mathcal{Q}_k \right\|_{L^\infty(\zeta)} d(x,y) \le \left\|\nabla \mathcal{Q}_k \right\|_{L^\infty(\X)} d(x,y) \:.
\end{equation*}
Now as by construction, $\mathcal{Q}_k$ is a spherical polynomial of degree at most $2^k$, we have by combining Lemma \ref{bernstein} and (\ref{cesaro_pf_eq5}),
\begin{equation} \label{cesaro_pf_eq6}
|\mathcal{Q}_k(x) - \mathcal{Q}_k(y)| \lesssim  2^k \|\mathcal{Q}_k\|_{L^\infty(\X)} \:d(x,y) \lesssim 2^{k(1-\gamma)}\:d(x,y)\:.
\end{equation}
Finally, by (\ref{cesaro_pf_eq4}), we write
\begin{equation*}
|f(x)-f(y)| \le \left|\mathcal{P}_1(x)-\mathcal{P}_1(y)\right|\;+\:\sum_{k=1}^\infty \left|\mathcal{Q}_k(x)-\mathcal{Q}_k(y)\right|\:.
\end{equation*}
As $\mathcal{P}_1$ is a spherical polynomial (of degree at most $1$), it is $C^\infty$ and hence again by the Hopf-Rinow theorem and the mean-value inequality, $\mathcal{P}_1$ is Lipschitz on $\X$. Combining this with the estimates (\ref{cesaro_pf_eq5}) and (\ref{cesaro_pf_eq6}), we obtain
\begin{eqnarray*}
|f(x)-f(y)| & \lesssim & d(x,y)^\gamma \:+\: \sum_{k=1}^\infty \min\left(2^{-k\gamma},2^{k(1-\gamma)}\:\pi^{-1}\:d(x,y)\right)\\
&\lesssim& d(x,y)^\gamma \:+\: \sum_{k=1}^{k_0-1} 2^{k(1-\gamma)}\:d(x,y)\:+\: \sum_{k=k_0}^\infty 2^{-k\gamma}\\
&\lesssim &   d(x,y)^\gamma\:.
\end{eqnarray*}
This completes the proof of Theorem \ref{thm_cesaro}.
\end{proof}

The proof of the converse part yields the following useful sufficiency condition:
\begin{corollary} \label{sufficiency_corollary}
If there exists a sequence of spherical polynomials $\mathcal{P}_{2^k}$ of degree $\le 2^k, k \in \N \cup \{0\}$, with $\|f-\mathcal{P}_{2^k}\|_{L^\infty(\X)} \lesssim 2^{-k\gamma}\:,\gamma \in (0,1),$ then $f \in C^\gamma(\X)$.
\end{corollary}

\section{Some applications}
\label{sec4}
In this section, we prove Corollaries \ref{corollary1}, \ref{corollary2} and \ref{corollary3}. In order to so, let us recall the following results from approximation theory:

\begin{lemma} \cite[Theorem 5.1]{Li} \label{Li_lemma1}
Let $f \in C(\X),\:\gamma \in (0,1)$ and $\delta>\frac{d-1}{2}$. Then 
\begin{equation*}
\left\|C^\delta_Nf-f\right\|_{L^{\infty}(X)} \lesssim N^{-\gamma}\:,\text{ as } N \to \infty\:,
\end{equation*}
if and only if 
\begin{equation*}
\sup_{0<r \le t} \|(\tau_r -I)^\frac{1}{2}f\|_{L^\infty(\X)} \lesssim t^\gamma\:,\text{ as } t \to 0+\:.
\end{equation*}
\end{lemma}

\begin{lemma} \cite[Theorem 5.2]{Li} \label{Li_lemma2}
Let $f \in C(\X),\:\gamma \in (0,1)$ and $\delta>\frac{d-1}{2}$. Then 
\begin{equation*}
\left\|C^\delta_Nf-f\right\|_{L^{\infty}(X)} \lesssim N^{-\gamma}\:,\text{ as } N \to \infty\:,
\end{equation*}
if and only if 
\begin{equation*}
\|S^\delta_Rf-f\|_{L^\infty(\X)} \lesssim R^{-\gamma}\:,\text{ as } R \to \infty\:.
\end{equation*}
\end{lemma}

\begin{proof}[Proof of Corollaries \ref{corollary1} and \ref{corollary2}] We note that $\delta_* > \frac{d-1}{2}$. The Corollaries \ref{corollary1} and \ref{corollary2} then follow at once from Theorem \ref{thm_cesaro} combined with Lemmata \ref{Li_lemma1} and \ref{Li_lemma2} respectively.
\end{proof}

\begin{proof}[Proof of Corollary \ref{corollary3}] 
We recall the sequence of partial sums defined in (\ref{partial_sum}) and consider the dyadic sub-collection $\{S_{2^k}\}_{k=0}^\infty$. Then it is a sequence of spherical polynomials of degree $\le 2^k$. Now if $f \in B^\gamma_{\infty, \infty}(\X)$, it follows from the definition that the series $\sum_{k=0}^\infty S_{2^k}$ is uniformly convergent to $f$ and moreover,
\begin{eqnarray*}
\left\|f-S_{2^k}f\right\|_{L^\infty(\X)} &=& \left\|f-\left(S_1f\:+\:\sum_{j=0}^{k-1}\left(S_{2^{j+1}}f-S_{2^j}f\right)\right)\right\|_{L^\infty(\X)}\\
&=& \left\|\sum_{j=k}^\infty \left(S_{2^{j+1}}f-S_{2^j}f\right)\right\|_{L^\infty(\X)}\\
&\le & \sum_{j=k}^\infty \left\|P_{2^j}f\right\|_{L^\infty(\X)} \\
& \lesssim & \sum_{j=k}^\infty 2^{-j\gamma}\\
& \lesssim & 2^{-k\gamma}\:.
\end{eqnarray*}
The result now follows from Corollary \ref{sufficiency_corollary}.
\end{proof}

\section{Upper Minkowski dimension of graphs of H\"older functions}
\label{sec5}
In this section, we prove the following result about the upper Minkowski dimension (defined in (\ref{minkowski_defn})) of  graphs of real-valued H\"older functions:
\begin{theorem} \label{thm_gmt}
Let $(\mathcal{X},d)$ be a compact metric space such that there are finitely many subsets $\Omega_j \subset \mathcal{X},\:1 \le j \le n$ with
\begin{itemize}
\item $\mathcal{X} = \displaystyle\bigcup_{j=1}^n \Omega_j$ and
\item there are bi-Lipschitz maps $\varphi_j : \Omega_j \to \varphi_j\left(\Omega_j\right) \subset \mathbb{R}^m\:,\: 1\le j \le n$. 
\end{itemize}
Then for a real-valued $f \in C^\gamma(\mathcal{X}),\:\gamma \in (0,1)$, we have $\overline{\dim}_M\left(Graph(f)\right) \le (m+1)-\gamma$\:.
\end{theorem}

Before proving Theorem \ref{thm_gmt}, we first see it in the simplest case of a bounded subset of $\R^m$:
\begin{lemma} \label{gmt_lemma3}
Let $E$ be a bounded subset of $\R^m$ and $f \in C^\gamma(E),\: \gamma \in (0,1)$, be real-valued. Then 
$$\overline{\dim}_M\left(Graph(f)\right) \le (m+1)-\gamma\:.$$
\end{lemma}
\begin{proof}
We first recall that to estimate the upper Minkowski dimension of a bounded subset $F \subset \R^{m+1}$, it is equivalent to consider $\mathcal{N}_\varepsilon(F)$ as the smallest number of boxes of sidelength $\le \varepsilon$ that can cover $F$ \cite[Section 3.1, Chapter 3]{Falconer}. For $\varepsilon>0$, we first cover $E$ by $m$-dimensional boxes with sidelength $\le \varepsilon$. The number of such boxes will be $\lesssim \varepsilon^{-m}$. Next as $f \in C^\gamma(E)$, we note that the range of $f$ on each such box is $\lesssim \varepsilon^\gamma$. To cover this range, the required number of intervals with length $\varepsilon$ will be $\lesssim \varepsilon^{\gamma -1}$. Now stacking these intervals on top of of the $m$-dimensional boxes, we get that the required number of $(m+1)$-dimensional boxes of sidelength $\le \varepsilon$ is,
\begin{equation*}
\mathcal{N}_\varepsilon(Graph(f)) \lesssim \varepsilon^{-m + \gamma -1}\:.
\end{equation*}
Plugging the above estimate in the definition (\ref{minkowski_defn}), the result follows.
\end{proof}

The idea of the proof of Theorem \ref{thm_gmt} is to utilize the Euclidean result (Lemma \ref{gmt_lemma3}) locally. In order to do so, we need the following tools:

\begin{lemma} \label{gmt_lemma1}
Let $\varphi$ be a bijective bi-Lipschitz map from $(\mathcal{X}_1,d_1) \to (\mathcal{X}_2,d_2)$ where both are totally bounded metric spaces. Then $\overline{\dim}_M\left(\mathcal{X}_1\right)=\overline{\dim}_M\left(\mathcal{X}_2\right)$\:.
\end{lemma}
\begin{proof}
By symmetry, it suffices to prove that $\varphi$ is Lipschitz implies $$\overline{\dim}_M\left(\mathcal{X}_2\right) \le \overline{\dim}_M\left(\mathcal{X}_1\right)\:.$$
By definition, there exists $C>0$ such that
\begin{equation*}
d_2\left(\varphi(x),\varphi(y)\right) \le C d_1(x,y)\:,\:x,y \in \mathcal{X}_1\:.
\end{equation*}
Thus for $\varepsilon>0$, any cover of $\mathcal{X}_1$ by balls of radii $\le \varepsilon$ yields a cover of $\mathcal{X}_2$ by balls of radii $\le C\varepsilon$. Hence, we have
\begin{equation} \label{gmt_lemma1_pf_eq1}
\mathcal{N}_{C\varepsilon}\left(\mathcal{X}_2\right) \le \mathcal{N}_{\varepsilon}\left(\mathcal{X}_1\right)\:,
\end{equation}
where we recall that $\mathcal{N}_{\delta}\left(\mathcal{X}\right)$ denotes the smallest number of metric balls of radii $\le \delta$ which can cover $\mathcal{X}$. Therefore,
\begin{equation*}
\overline{\dim}_M\left(\mathcal{X}_2\right) = \limsup_{\varepsilon \to 0} \frac{\log\left(\mathcal{N}_{C\varepsilon}\left(\mathcal{X}_2\right) \right)}{\log\left(\frac{1}{\varepsilon}\right)\left[1+\frac{\log\left(\frac{1}{C}\right)}{\log\left(\frac{1}{\varepsilon}\right)}\right]}\:.
\end{equation*}
Now as 
\begin{equation*}
\frac{\log\left(\frac{1}{C}\right)}{\log\left(\frac{1}{\varepsilon}\right)} \to 0\:,\:\:\text{ as } \varepsilon \to 0\:,
\end{equation*}
combining with (\ref{gmt_lemma1_pf_eq1}), we get the desired inequality,
$$\overline{\dim}_M\left(\mathcal{X}_2\right) \le \limsup_{\varepsilon \to 0} \frac{\log\left(\mathcal{N}_{\varepsilon}\left(\mathcal{X}_1\right) \right)}{\log\left(\frac{1}{\varepsilon}\right)} = \overline{\dim}_M\left(\mathcal{X}_1\right)\:.$$
This completes the proof of Lemma \ref{gmt_lemma1}.
\end{proof}

\begin{lemma} \label{gmt_lemma2}
Let  $\left(\mathcal{X},d\right)$ be a totally bounded metric space and $\Omega_j \subset \mathcal{X},\:1\le j \le n$. Then 
\begin{equation*}
\overline{\dim}_M\left(\bigcup_{j=1}^n\Omega_j\right) = \max_{1 \le j \le n}\overline{\dim}_M\left(\Omega_j\right)\:.
\end{equation*} 
\end{lemma}
\begin{proof}
Clearly, $$\overline{\dim}_M\left(\bigcup_{j=1}^n\Omega_j\right) \ge \overline{\dim}_M\left(\Omega_j\right),\:\text{ for all } 1\le j \le n\:,$$
and thus
\begin{equation} \label{gmt_lemma2_pf_eq1}
\overline{\dim}_M\left(\bigcup_{j=1}^n\Omega_j\right) \ge \max_{1 \le j \le n}\overline{\dim}_M\left(\Omega_j\right)\:.
\end{equation}
On the other hand for $\varepsilon>0$, as union of individual covers of $\Omega_j$ by balls of radii $\le \varepsilon$, constitute a cover of $\bigcup_{j=1}^n\Omega_j$, we have
\begin{equation*}
\mathcal{N}_{\varepsilon}\left(\bigcup_{j=1}^n\Omega_j\right) \le \sum_{j=1}^n \mathcal{N}_{\varepsilon}\left(\Omega_j\right) \le n \max_{1 \le j \le n} \mathcal{N}_{\varepsilon}\left(\Omega_j\right)\:.
\end{equation*}
Hence,
\begin{eqnarray*}
\overline{\dim}_M\left(\bigcup_{j=1}^n\Omega_j\right) &=& \limsup_{\varepsilon \to 0} \frac{\log\left(\mathcal{N}_{\varepsilon}\left(\bigcup_{j=1}^n\Omega_j\right)\right)}{\log\left(\frac{1}{\varepsilon}\right)} \\
& \le & \limsup_{\varepsilon \to 0} \frac{\log\left(n \displaystyle\max_{1 \le j \le n} \mathcal{N}_{\varepsilon}\left(\Omega_j\right)\right)}{\log\left(\frac{1}{\varepsilon}\right)} \\
& \le & \limsup_{\varepsilon \to 0} \frac{\log n}{\log\left(\frac{1}{\varepsilon}\right)} \:+\: \limsup_{\varepsilon \to 0} \frac{\log\left( \displaystyle\max_{1 \le j \le n} \mathcal{N}_{\varepsilon}\left(\Omega_j\right)\right)}{\log\left(\frac{1}{\varepsilon}\right)} \:.
\end{eqnarray*}
As
\begin{equation*}
\frac{\log n}{\log\left(\frac{1}{\varepsilon}\right)} \to 0\:,\:\text{ as } \varepsilon \to 0\:,
\end{equation*}
we get
\begin{equation}\label{gmt_lemma2_pf_eq2}
\overline{\dim}_M\left(\bigcup_{j=1}^n\Omega_j\right) \le \max_{1 \le j \le n} \limsup_{\varepsilon \to 0} \frac{\log\left( \mathcal{N}_{\varepsilon}\left(\Omega_j\right)\right)}{\log\left(\frac{1}{\varepsilon}\right)} = \max_{1 \le j \le n} \overline{\dim}_M\left(\Omega_j\right)\:.
\end{equation}
Combining (\ref{gmt_lemma2_pf_eq1}) and (\ref{gmt_lemma2_pf_eq2}), we get the result.
\end{proof}

We are now in a position to prove Theorem \ref{thm_gmt}.
\begin{proof}[Proof of Theorem \ref{thm_gmt}]
We set for $1 \le j \le n$,
\begin{equation*}
G_j:= \left\{\left(x,f(x)\right) : x \in \Omega_j\right\} \subset Graph(f)\:.
\end{equation*}
Then
\begin{equation} \label{thm_gmt_pf_eq1}
Graph(f) = \bigcup_{j=1}^n G_j\:.
\end{equation}
Now as $\mathcal{X}$ is compact and $f$ is continuous, we have that $Graph(f)$ is a compact subset of $\mathcal{X} \times \R$ and consequently all the above subsets $G_j$ are totally bounded. 

Next for each $1 \le j \le n$, we set $\tilde{G}_j$ to be the graph of  $f \circ \varphi^{-1}_j$ on $\varphi_j\left(\Omega_j\right)$, that is,
\begin{equation*}
\tilde{G}_j:=\left\{\left(y,f \circ \varphi^{-1}_j(y)\right) : y \in \varphi_j\left(\Omega_j\right)\right\}\:.
\end{equation*}
We next define,
\begin{eqnarray*}
\Phi_j : G_j &\to & \tilde{G}_j \\
\left(x,f(x)\right) &\mapsto & \left(\varphi_j(x),f(x)\right)\:.
\end{eqnarray*}
As $\varphi_j$ is a bi-Lipschitz map from $\left(\Omega_j,d\right)$ onto $\left(\varphi_j\left(\Omega_j\right),d_{\R^m}\right)$ ($d_{\R^m}$ denotes the canonical metric on $\R^m$), it follows that $\Phi_j$ is a bi-Lipschitz map from $\left(G_j,\:d_{\mathcal{X} \times \R}\right)$ onto  $\left(\tilde{G}_j,\:d_{\R^{m+1}}\right)$, where $d_{\mathcal{X} \times \R}$ denotes the restriction of the product metric on $\mathcal{X} \times \R$ to $G_j$. Thus $\tilde{G}_j$ being the image of a totally bounded set $G_j$ under $\Phi_j$ is also totally bounded and moreover, by Lemma \ref{gmt_lemma1},
\begin{equation} \label{thm_gmt_pf_eq2}
\overline{\dim}_M\left(G_j\right)=\overline{\dim}_M\left(\tilde{G}_j\right)\:.
\end{equation}
Next we note that as $f \in C^\gamma(\mathcal{X})$, we also have  $f \circ \varphi^{-1}_j \in C^\gamma(\varphi_j\left(\Omega_j\right))$ and hence, $\tilde{G}_j$ is the graph of a real-valued $C^\gamma$ function on a bounded set $\varphi_j\left(\Omega_j\right) \subset \R^m$. Then by Lemma \ref{gmt_lemma3}, we have that
\begin{equation} \label{thm_gmt_pf_eq3}
\overline{\dim}_M\left(\tilde{G}_j\right) \le (m+1)-\gamma\:.
\end{equation}
Finally, applying Lemma \ref{gmt_lemma2} on (\ref{thm_gmt_pf_eq1}) and then combining it with (\ref{thm_gmt_pf_eq2}) and (\ref{thm_gmt_pf_eq3}), we get the desired bound,
\begin{equation*}
\overline{\dim}_M\left(Graph(f)\right) =\max_{1 \le j \le n}\overline{\dim}_M\left(G_j\right)=\max_{1 \le j \le n}\overline{\dim}_M\left(\tilde{G}_j\right) \le (m+1)-\gamma\:.
\end{equation*}
\end{proof}

We have the following important special case of Theorem \ref{thm_gmt}, which will be crucial for us later:
\begin{corollary} \label{graph_corollary}
Let $\mathcal{X}$ be a $d$-dimensional compact, connected Riemannian manifold  and $f \in C^\gamma(\mathcal{X}),\: \gamma \in (0,1)$, be real-valued. Then 
$$\overline{\dim}_M\left(Graph(f)\right) \le (d+1)-\gamma\:.$$
\end{corollary}
\begin{proof}
Corollary \ref{graph_corollary} follows at once from Theorem \ref{thm_gmt} by combining the fact that chart maps on a Riemannian manifold are locally bi-Lipschitz and the underlying compactness of $\mathcal{X}$.
\end{proof}

\section{Oscillatory expansions of zonal spherical functions}
\label{sec6}
In this section, our main result is the following oscillatory expansion of the zonal spherical functions of degree $n \ge 1$, on a rank one Riemannian symmetric space of compact type:
\begin{theorem} \label{thm_oscillatory_zn}
\begin{enumerate}
\item \label{part1} On $P^d(\R)$, for integers $m \ge 1$ and $\theta \in [0,\pi]$, we have
\begin{equation*}
Z_n(\theta)=\psi^+_n(\theta)e^{in\frac{\theta}{2}}\:+\:\psi^-_n(\theta)e^{-in\frac{\theta}{2}}\:+\:E(\theta,n)\:,
\end{equation*}
where
\begin{equation*}
\left|\psi^\pm_n(\theta)\right|\lesssim \frac{n^{d-1}}{\langle n \theta \rangle^{\frac{d-1}{2}}}\:,\:\: \left|\psi^\pm_n(\theta)-\psi^\pm_{n-1}(\theta)\right|\lesssim \frac{n^{d-2}}{\langle n \theta \rangle^{\frac{d-1}{2}}}\:,\:\text{ and } \left|E(\theta,n)\right| \lesssim n^{d-(m+1)}\:. 
\end{equation*}
\item \label{part2} On $\Sb^d,P^d(\C),P^d(\Hb)$ and $P^{16}(Cay)$, for  integers  $m \ge 1$ and $\theta \in \left[0,\frac{\pi}{2}\right]$, we have
\begin{equation*}
Z_n(\theta)=\psi^+_n(\theta)e^{in\theta}\:+\:\psi^-_n(\theta)e^{-in\theta}\:+\:E_1(\theta,n)\:,
\end{equation*}
where
\begin{equation*}
\left|\psi^\pm_n(\theta)\right|\lesssim \frac{n^{d-1}}{\langle n \theta \rangle^{\frac{d-1}{2}}}\:,\:\: \left|\psi^\pm_n(\theta)-\psi^\pm_{n-1}(\theta)\right|\lesssim \frac{n^{d-2}}{\langle n \theta \rangle^{\frac{d-1}{2}}}\:,\:\text{ and } \left|E_1(\theta,n)\right| \lesssim n^{d-(m+1)}\:. 
\end{equation*}
\item \label{part3} On $\Sb^d,P^d(\C),P^d(\Hb)$ and $P^{16}(Cay)$, for  positive integers $m >\frac{d-2(1+\beta)}{4}$ (where $\beta$ is as in Table \ref{table:2}) and $\theta \in \left[\frac{\pi}{2},\pi\right]$, we have
\begin{equation*}
Z_n(\theta)=\lambda^+_n(\theta)e^{in\theta}\:+\:\lambda^-_n(\theta)e^{-in\theta}\:+\:E_2(\theta,n)\:,
\end{equation*}
where
\begin{eqnarray*}
&&\left|\lambda^\pm_n(\theta)\right|\lesssim \frac{n^{\beta + \frac{d}{2}}}{\langle n (\pi-\theta) \rangle^{\beta + \frac{1}{2}}}\:,\:\: \left|\lambda^\pm_n(\theta)-\lambda^\pm_{n-1}(\theta)\right|\lesssim \frac{n^{\left(\beta + \frac{d}{2}\right)-1}}{\langle n (\pi-\theta) \rangle^{\beta + \frac{1}{2}}}\:,\\
&&\text{ and } \left|E_2(\theta,n)\right| \lesssim n^{\left(\beta + \frac{d}{2}\right)-m}\:. 
\end{eqnarray*}
\end{enumerate}
\end{theorem}

Before proceeding with the proof of Theorem \ref{thm_oscillatory_zn}, we make the following remark:
\begin{remark} \label{epxansion_symmetry_sphere}
In the case of $\Sb^d$, since $s=1$ and $\alpha=\beta=(d-2)/2$, the expansions of $Z_n$ in parts (\ref{part2}) and (\ref{part3}) of Theorem \ref{thm_oscillatory_zn} are identical. This can be attributed to the fact that the antipodal manifold/cut locus corresponding to any point is singleton and hence from the oscillatory expansion view-point, both the origin and the cut locus behave similarly.
\end{remark}

In view of (\ref{in_terms_jacobi}), to prove Theorem \ref{thm_oscillatory_zn}, we focus on the oscillatory expansions of the corresponding Jacobi polynomials. But first, we recall the following result for Bessel functions of the first kind (see \cite{Szego} for more details) which follows by relating Bessel functions to Fourier transforms of surface measures on spheres (see \cite[p. 7]{MathZ}): 
\begin{lemma} \label{Bessel_function_expansion}
We have for integers $d \ge 2$,
\begin{equation*}
(2\pi)^{\frac{d}{2}} \frac{J_{\frac{d-2}{2}}(r)}{r^\frac{d-2}{2}} = \omega^+_{d-1}(r)e^{ir}\:+\:\omega^-_{d-1}(r)e^{-ir}\:,\:\:r \ge 0\:,
\end{equation*}
where
\begin{equation} \label{coefficient_estimates}
\left|\omega^\pm_{d-1}(r)\right| \lesssim \langle r\rangle^{-\frac{d-1}{2}}\:,\:\:\text{ and } 
\left|\frac{\partial}{\partial r}\omega^\pm_{d-1}(r)\right| \lesssim \langle r\rangle^{-\frac{d+1}{2}}\:.
\end{equation}
\end{lemma}
Setting the notation,
\begin{equation*}
\frac{1}{2}\N \cup\{0\}:=\left\{\frac{k-2}{2} \mid k \ge 2,\:k\in \N\right\}\:,
\end{equation*}
we now present the oscillatory expansions of Jacobi polynomials:
\begin{lemma} \label{jacobi_lemma1}
Let $\alpha,\beta \in \frac{1}{2}\N \cup \{0\}$. Then for integers $n \ge 1$, $m>\max \left\{0,\frac{\beta - \alpha}{2}\right\}$ and $\theta \in \left[0,\frac{\pi}{2}\right]$, we have
\begin{equation*}
P^{(\alpha, \beta)}_n(\cos \theta)=\Psi^+_n(\theta)e^{in\theta}\:+\:\Psi^-_n(\theta)e^{-in\theta}\:+\:\mathcal{E}_1(\theta,n)\:,
\end{equation*}
where
\begin{equation*}
\left|\Psi^\pm_n(\theta)\right|\lesssim \frac{n^\alpha}{\langle n \theta \rangle^{\alpha+\frac{1}{2}}}\:,\:\: \left|\Psi^\pm_n(\theta)-\Psi^\pm_{n-1}(\theta)\right|\lesssim \frac{n^{\alpha-1}}{\langle n \theta \rangle^{\alpha +\frac{1}{2}}}\:,\:\text{ and } \left|\mathcal{E}_1(\theta,n)\right| \lesssim n^{\alpha-m}\:. 
\end{equation*}
\end{lemma}
\begin{proof}
By Lemma \ref{frenzen}, for positive integers $m>\frac{\beta - \alpha}{2}$ and $\theta \in \left[0,\frac{\pi}{2}\right]$ we have,
\begin{equation*}
P^{(\alpha,\beta)}_n(\cos \theta)= \frac{\Gamma(n+\alpha+1)}{n!} \left[\sum_{l=0}^{m-1}B_l(\theta)\frac{J_{\alpha+l}(N\theta)}{{\left(N\theta\right)}^{\alpha + l}} + \mathcal{O}\left(\frac{\theta^m}{N^m}\right)\right]\:,
\end{equation*}
where
\begin{equation*}
\left|B_l(\theta)\right| \lesssim \theta^{2l}\:,\text{ and }N=n+\frac{\alpha + \beta +1}{2}\:.
\end{equation*}
Then plugging in Lemma \ref{Bessel_function_expansion} in the above, we get
\begin{eqnarray*}
P^{(\alpha,\beta)}_n(\cos \theta)=&& \frac{\Gamma(n+\alpha+1)}{n!}e^{i\frac{(\alpha + \beta +1)\theta}{2}}\left[\sum_{l=0}^{m-1}
\tilde{B}_l(\theta)\omega^+_{2(\alpha + l)+1}(N\theta)\right]e^{in\theta} \\
&+& \frac{\Gamma(n+\alpha+1)}{n!}e^{-i\frac{(\alpha + \beta +1)\theta}{2}}\left[\sum_{l=0}^{m-1}
\tilde{B}_l(\theta)\omega^-_{2(\alpha + l)+1}(N\theta)\right]e^{-in\theta} \\
&+& \frac{\Gamma(n+\alpha+1)}{n!}\mathcal{O}\left(\frac{\theta^m}{N^m}\right)\:.
\end{eqnarray*}
Now setting
\begin{eqnarray*}
&&\Psi^\pm_n(\theta) :=\frac{\Gamma(n+\alpha+1)}{n!}e^{\pm i\frac{(\alpha + \beta +1)\theta}{2}}\left[\sum_{l=0}^{m-1}
\tilde{B}_l(\theta)\omega^\pm_{2(\alpha + l)+1}\left(\left(n+\frac{\alpha + \beta +1}{2}\right)\theta\right)\right]\:,\\
&&\mathcal{E}_1(\theta,n):=\frac{\Gamma(n+\alpha+1)}{n!}\mathcal{O}\left(\frac{\theta^m}{N^m}\right)\:,
\end{eqnarray*}
we are left to prove the estimates. At this juncture, we recall by Stirling's approximation that
\begin{equation} \label{Stirling}
\frac{\Gamma(n+\alpha+1)}{n!} \asymp n^\alpha\:,\:n \in \N\:.
\end{equation}
Thus for the remainder term, we get
\begin{equation*}
\left|\mathcal{E}_1(\theta,n)\right| \lesssim n^{\alpha -m}\:.
\end{equation*}
Next we have by (\ref{coefficient_estimates}) and (\ref{Stirling}),
\begin{eqnarray*}
\left|\Psi^\pm_n(\theta)\right| &\lesssim & n^\alpha \sum_{l=0}^{m-1} \theta^{2l} \langle n\theta\rangle^{-(\alpha +l)-\frac{1}{2}} \\
&=& \frac{n^\alpha}{\langle n\theta\rangle^{\alpha + \frac{1}{2}}} \sum_{l=0}^{m-1} \theta^{2l} \langle n\theta\rangle^{-l} \\
&\lesssim & \frac{n^\alpha}{\langle n\theta\rangle^{\alpha + \frac{1}{2}}}\:.
\end{eqnarray*}
We finally look at the difference of the succesive terms,
\begin{eqnarray*}
&&|\Psi^\pm_n(\theta) -\Psi^\pm_{n-1}(\theta)| \\
&\lesssim& \sum_{l=0}^{m-1} \theta^{2l}\\
&\times&\left|\frac{\Gamma(n+\alpha+1)}{n!}\omega^\pm_{2(\alpha + l)+1}\left(\left(n+\frac{\alpha + \beta +1}{2}\right)\theta\right)-\frac{\Gamma(n+\alpha)}{(n-1)!}\omega^\pm_{2(\alpha + l)+1}\left(\left(n+\frac{\alpha + \beta -1}{2}\right)\theta\right)\right| \\
&\le & \sum_{l=0}^{m-1} \theta^{2l} \frac{\Gamma(n+\alpha+1)}{n!} \left|\omega^\pm_{2(\alpha + l)+1}\left(\left(n+\frac{\alpha + \beta +1}{2}\right)\theta\right)- \omega^\pm_{2(\alpha + l)+1}\left(\left(n+\frac{\alpha + \beta -1}{2}\right)\theta\right)\right| \\&&+ \sum_{l=0}^{m-1} \theta^{2l}\left|\frac{\Gamma(n+\alpha+1)}{n!}-\frac{\Gamma(n+\alpha)}{(n-1)!}\right|\left|\omega^\pm_{2(\alpha + l)+1}\left(\left(n+\frac{\alpha + \beta -1}{2}\right)\theta\right)\right|\:.
\end{eqnarray*}
Then by (\ref{Stirling}), the mean-value inequality and the derivative bound in (\ref{coefficient_estimates}), we estimate the first sum,
\begin{eqnarray*}
&&\sum_{l=0}^{m-1} \theta^{2l} \frac{\Gamma(n+\alpha+1)}{n!} \left|\omega^\pm_{2(\alpha + l)+1}\left(\left(n+\frac{\alpha + \beta +1}{2}\right)\theta\right)- \omega^\pm_{2(\alpha + l)+1}\left(\left(n+\frac{\alpha + \beta -1}{2}\right)\theta\right)\right| \\
&\lesssim & \sum_{l=0}^{m-1} \theta^{2l}\: n^\alpha\: \theta\: \langle n\theta \rangle^{-(\alpha +l)-\frac{3}{2}}\\
&\le & \frac{n^{\alpha -1}}{\langle n\theta\rangle^{\alpha + \frac{1}{2}}} \sum_{l=0}^{m-1} \theta^{2l} \: \langle n\theta\rangle^{-l} \\
& \lesssim & \frac{n^{\alpha -1}}{\langle n\theta\rangle^{\alpha + \frac{1}{2}}}\:. 
\end{eqnarray*}
On the other hand, by utilizing the fundamental property of Gamma function: $\Gamma(z+1)=z\Gamma(z)$, the Stirling's approximation and the pointwise bound in (\ref{coefficient_estimates}), we get for the second sum
\begin{eqnarray*}
&&\sum_{l=0}^{m-1} \theta^{2l}\left|\frac{\Gamma(n+\alpha+1)}{n!}-\frac{\Gamma(n+\alpha)}{(n-1)!}\right|\left|\omega^\pm_{2(\alpha + l)+1}\left(\left(n+\frac{\alpha + \beta -1}{2}\right)\theta\right)\right| \\
&\lesssim & n^{\alpha -1} \sum_{l=0}^{m-1} \theta^{2l} \langle n\theta\rangle^{-(\alpha +l)-\frac{1}{2}} \\
&=& \frac{n^{\alpha -1}}{\langle n\theta\rangle^{\alpha + \frac{1}{2}}} \sum_{l=0}^{m-1} \theta^{2l} \langle n\theta\rangle^{-l} \\
&\lesssim & \frac{n^\alpha}{\langle n\theta\rangle^{\alpha + \frac{1}{2}}}\:.
\end{eqnarray*}
Thus we conclude,
\begin{equation*}
|\Psi^\pm_n(\theta) -\Psi^\pm_{n-1}(\theta)| \lesssim  \frac{n^{\alpha -1}}{\langle n\theta\rangle^{\alpha + \frac{1}{2}}}\:.
\end{equation*}
This completes the proof of Lemma \ref{jacobi_lemma1}.
\end{proof}

\begin{lemma} \label{jacobi_lemma2}
Let $\alpha,\beta \in \frac{1}{2}\N \cup \{0\}$. Then for integers $n \ge 1,\:m>\max \left\{0,\frac{\alpha - \beta}{2}\right\}$ and $\theta \in \left[\frac{\pi}{2},\pi\right]$, we have
\begin{equation*}
P^{(\alpha, \beta)}_n(\cos \theta)=\Lambda^+_n(\theta)e^{in\theta}\:+\:\Lambda^-_n(\theta)e^{-in\theta}\:+\:\mathcal{E}_2(\theta,n)\:,
\end{equation*}
where
\begin{equation*}
\left|\Lambda^\pm_n(\theta)\right|\lesssim \frac{n^\beta}{\langle n (\pi-\theta) \rangle^{\beta+\frac{1}{2}}}\:,\:\: \left|\Lambda^\pm_n(\theta)-\Lambda^\pm_{n-1}(\theta)\right|\lesssim \frac{n^{\beta-1}}{\langle n (\pi-\theta) \rangle^{\beta+\frac{1}{2}}}\:,\:\text{ and } \left|\mathcal{E}_2(\theta,n)\right| \lesssim n^{\beta-m}\:. 
\end{equation*}
\end{lemma}
\begin{proof}
By the change of variable, $\eta=\pi-\theta \in \left[0,\frac{\pi}{2}\right]$ and the fundamental symmetry of Jacobi polynomials (\ref{jacobi_symmetry}), we have
\begin{equation*}
P^{(\alpha,\beta)}_n(\cos \theta)=P^{(\alpha,\beta)}_n(\cos (\pi-\eta))=P^{(\alpha,\beta)}_n(-\cos \eta)=(-1)^n P^{(\beta,\alpha)}_n(\cos \eta)\:.
\end{equation*}
The result then follows by repeating the proof of Lemma \ref{jacobi_lemma1} for $P^{(\beta,\alpha)}_n$.
\end{proof}

We now complete the proof of Theorem \ref{thm_oscillatory_zn}:
\begin{proof}[Proof of Theorem \ref{thm_oscillatory_zn}]
As mentioned earlier, the key identity is (\ref{in_terms_jacobi}):
\begin{equation*} 
Z_n(\theta):= d_n \frac{P^{(\alpha,\beta)}_n{\left(\cos\left(\frac{\theta}{s}\right)\right)}}{P^{(\alpha,\beta)}_n(1)}\:,\:\:\theta \in [0,\pi]\:.
\end{equation*}
We also note by (\ref{dimension_growth}) and (\ref{jacobi_at_1}) that for $n \in \N$,
\begin{equation} \label{thm_zonal_pf_eq1}
\frac{d_n}{P^{(\alpha,\beta)}_n(1)} \asymp n^{\frac{d}{2}}\:. 
\end{equation} 
For part (\ref{part1}), we note that in the case of $P^d(\R)$, $s=2$ and also $\alpha=\beta=\frac{d-2}{2}$. Hence the result follows from Lemma \ref{jacobi_lemma1} and (\ref{thm_zonal_pf_eq1}). Similarly parts (\ref{part2}) and (\ref{part3}) follow from Lemmata \ref{jacobi_lemma1} and \ref{jacobi_lemma2} respectively, combined with (\ref{thm_zonal_pf_eq1}) and the fact that $\alpha=\frac{d-2}{2}$ always.
\end{proof}

\section{Talbot effect} \label{sec7}
In this section, we will prove Theorem \ref{thm_talbot} and Corollary \ref{talbot_corollary}. But first we recall the following Weyl sum estimate:
\begin{lemma}\cite[Lemma 3.2]{MathZ}  \label{Weyl_estimate}
Let $\{b_n\}_{n \ge 1}$ be a sequence of complex numbers such that
\begin{equation*}
|b_n| \lesssim n^p\:\:\text{ and } |b_n-b_{n-1}| \lesssim n^{p-1}\:,
\end{equation*}
for some $p \in \R$. Then for almost every $t$ and all $N \in \N$, we have
\begin{equation*}
\sup_{N \le u \le 2N} \sup_{x \in \R} \left|\sum_{n=N}^u e^{in^2t+inx}\: b_n\right|\lesssim_t N^{p+\frac{1}{2}}\:.
\end{equation*}
\end{lemma}

We are now ready to prove Theorem \ref{thm_talbot}:
\begin{proof}[Proof of Theorem \ref{thm_talbot}]
We recall that the real projective space $P^d(\R)$ can be obtained from $\Sb^d$ by identifying the antipodal points:
\begin{eqnarray*}
\pi : &&\Sb^d\:\: \to \:\: P^d(\R) \\
&& \pm x \:\:\mapsto \:\: [x]\:,
\end{eqnarray*}
with the projection map $\pi$ being a local isometry.  Thus the functions on $P^d(\R)$ can be identified as even functions on $\Sb^d$ and if $f_e$ is the even function on $\Sb^d$ corresponding to the function $f$ on $P^d(\R)$, then $\Delta_{P^d(\R)} f=\Delta_{\Sb^d}f_e$. Also the distances on $P^d(\R)$ and $\Sb^d$ are related as follows:
\begin{equation*}
d_{P^d(\R)}\left(\pi(x),\pi(y)\right)= \min \left\{d_{\Sb^d}(x,y)\:,\:d_{\Sb^d}(x,-y)\right\}\:,
\end{equation*}
and thus there is a one-to-one correspondence between the H\"older spaces $C^\gamma\left(P^d(\R)\right)$  and $C^\gamma\left(\Sb^d\right)_e$, the collection of even functions in $C^\gamma\left(\Sb^d\right),\:\gamma \in (0,1)$. Moreover, combining the facts that $\pi$ is a local isometry and the compactness of $P^d(\R)$, proceeding as in the proof of Theorem \ref{thm_gmt}, we note that the upper Minkowski dimension of sets are preserved under $\pi$. Therefore the validity of Theorem \ref{thm_talbot} on $P^d(\R)$ follows from that of $\Sb^d$. Hence, we will prove Theorem \ref{thm_talbot} for $\X=\Sb^d,P^d(\C),\:P^d(\Hb)$ and $P^{16}(Cay)$.

To prove part (\ref{talbot_part1}) of Theorem \ref{thm_talbot}, it suffices to prove that
\begin{equation} \label{thm_talbot_pf_eq1}
u(\cdot,t) \in B^{\gamma}_{\infty,\infty}(\X)\:,\:\text{ for almost all }t\:. 
\end{equation}
Indeed, by Corollary \ref{corollary3}, (\ref{thm_talbot_pf_eq1}) implies that for almost all $t,\:u(\cdot,t) \in C^{\gamma}(\X)$. 

Thus by the definition of $B^{\gamma}_{\infty,\infty}(\X)$, we need to show the following estimate for the Littlewood-Paley projections, for almost all $t$:
\begin{equation} \label{thm_talbot_pf_eq2}
\left\|P_Nu(\cdot,t)\right\|_{L^\infty(\X)} \lesssim N^{-\gamma}\:,\:N\ge 1,\:\text{ dyadic }\:.
\end{equation}

By orthogonality of projections, we can write $P_Nu(\cdot,t)$ as a right convolution with a radial kernel on $\X$:
\begin{equation*}
P_Nu(\cdot,t)=P_Nf * H_{N,t}\:,
\end{equation*}
where
\begin{equation} \label{thm_talbot_pf_eq3} 
H_{N,t}(\theta):= \sum_{n=N+1}^{2N} e^{itn(n+\alpha+\beta+1)}\:Z_n(\theta)\:,\:\:\theta \in [0,\pi]\:,
\end{equation}
where we recall that $\alpha,\beta$ are as in Table \ref{table:2} and $Z_n$ are the zonal spherical functions of degree $n$. Then by Young's inequality, we have
\begin{equation} \label{thm_talbot_pf_eq4}
\left\|P_Nu(\cdot,t)\right\|_{L^\infty(\X)} \le \left\|P_Nf\right\|_{L^p(\X)}\|H_{N,t}\|_{L^{p'}(\X)}\:,
\end{equation}
where $p'$ is the H\"older conjugate of $p$.

Now to estimate $\|H_{N,t}\|_{L^{p'}(\X)}$, we seek pointwise bounds utilizing Theorem \ref{thm_oscillatory_zn}. First for $\theta \in \left[0,\frac{\pi}{2}\right]$, by invoking part (\ref{part2}) of Theorem \ref{thm_oscillatory_zn} for $m \ge \frac{d}{2}$ into (\ref{thm_talbot_pf_eq3}), we get
\begin{eqnarray*}
\left|H_{N,t}(\theta)\right| &\le& \left|\sum_{n=N+1}^{2N} e^{itn(n+\alpha+\beta+1)}\:e^{in\theta}\psi^+_n(\theta)\right|\\
&&+ \left|\sum_{n=N+1}^{2N} e^{itn(n+\alpha+\beta+1)}\:e^{-in\theta}\psi^-_n(\theta)\right|\\
&&+\left|\sum_{n=N+1}^{2N} e^{itn(n+\alpha+\beta+1)}\:E_1(\theta, n)\right|\:.
\end{eqnarray*}
Then by the pointwise bounds of $\left|E_1(\theta, n)\right|$, we have
\begin{equation} \label{thm_talbot_pf_eq5}
\left|\sum_{n=N+1}^{2N} e^{itn(n+\alpha+\beta+1)}\:E_1(\theta, n)\right| \lesssim N^{d-m}\:.
\end{equation}
To estimate the other two Weyl sums, we consider two separate cases. First when $\theta \in \left[0,\frac{1}{N}\right]$, combining the estimates of $|\psi^\pm_n(\theta)|$ and $|\psi^\pm_n(\theta)-\psi^\pm_{n-1}(\theta)|$ (provided in part (\ref{part2}) of Theorem \ref{thm_oscillatory_zn}) and Lemma \ref{Weyl_estimate}, we get
\begin{equation} \label{thm_talbot_pf_eq6}
\left|\sum_{n=N+1}^{2N} e^{in^2t+in\left(t(\alpha+\beta+1)\pm \theta\right)}\:\psi^\pm_n(\theta)\right| \lesssim_t N^{d-\frac{1}{2}} \lesssim \frac{N^{d-\frac{1}{2}}}{\langle N\theta\rangle^{\frac{d-1}{2}}}\:.
\end{equation}
In the second case, when $\theta \in \left[\frac{1}{N},\frac{\pi}{2}\right]$, again combining the estimates of $|\psi^\pm_n(\theta)|$ and $|\psi^\pm_n(\theta)-\psi^\pm_{n-1}(\theta)|$ (provided in part (\ref{part2}) of Theorem \ref{thm_oscillatory_zn}) and Lemma \ref{Weyl_estimate}, we get
\begin{equation} \label{thm_talbot_pf_eq7}
\left|\sum_{n=N+1}^{2N} e^{in^2t+in\left(t(\alpha+\beta+1)\pm \theta\right)}\:\psi^\pm_n(\theta)\right| \lesssim_t \frac{N^{\frac{d}{2}}}{\theta^{\frac{d-1}{2}}} \lesssim \frac{N^{d-\frac{1}{2}}}{\langle N\theta\rangle^{\frac{d-1}{2}}}\:.
\end{equation}
Thus as $m \ge \frac{d}{2}$, by (\ref{thm_talbot_pf_eq5})-(\ref{thm_talbot_pf_eq7}), we get 
\begin{equation} \label{kernel_estimate1}
\left|H_{N,t}(\theta)\right| \lesssim_t \frac{N^{d-\frac{1}{2}}}{\langle N\theta\rangle^{\frac{d-1}{2}}}\:,\:\:\theta \in \left[0,\frac{\pi}{2}\right]\:.
\end{equation}
Then by the formula for the density (\ref{density}), it follows for $p'\ne \infty$,
\begin{eqnarray*}
&&\int_0^{\frac{\pi}{2}} \left|H_{N,t}(\theta)\right|^{p'}\mathcal{A}(\theta)\:d\theta \\
&\lesssim_t & N^{p'\left(d-\frac{1}{2}\right)}\int_0^{\frac{\pi}{2}} \frac{\theta^{d-1}}{\langle N\theta\rangle^{\frac{p'(d-1)}{2}}}\:d\theta\\
&=& N^{p'\left(d-\frac{1}{2}\right)}\int_0^{\frac{1}{N}} \frac{\theta^{d-1}}{\langle N\theta\rangle^{\frac{p'(d-1)}{2}}}\:d\theta\:+\:N^{p'\left(d-\frac{1}{2}\right)}\int_{\frac{1}{N}}^{\frac{\pi}{2}} \frac{\theta^{d-1}}{\langle N\theta\rangle^{\frac{p'(d-1)}{2}}}\:d\theta\\
&\lesssim& N^{p'\left(d-\frac{1}{2}\right)}\int_0^{\frac{1}{N}} \theta^{d-1}\:d\theta\:+\:N^{\frac{p'd}{2}}\int_{\frac{1}{N}}^{\frac{\pi}{2}} \theta^{\left(1-\frac{p'}{2}\right)(d-1)}\:d\theta\\
&\lesssim& N^{p'\left(d-\frac{1}{2}\right)-d}\:+\:N^{\frac{p'd}{2}}\int_{\frac{1}{N}}^{\frac{\pi}{2}} \theta^{\left(1-\frac{p'}{2}\right)(d-1)}\:d\theta\:.
\end{eqnarray*}
A simple calculation now yields,
\begin{eqnarray} \label{thm_talbot_pf_eq8}
\int_0^{\frac{\pi}{2}} \left|H_{N,t}(\theta)\right|^{p'}\mathcal{A}(\theta)\:d\theta &\lesssim_t & N^{p'\left(d-\frac{1}{2}\right)-d} + \begin{cases}
      N^{\frac{p'd}{2}}\:\:&\text{ if } 1 \le p' < \frac{2d}{d-1}\:,\\
	 N^{\frac{p'd}{2}+}\:\:&\text{ if } p' = \frac{2d}{d-1}\:,\\
	 N^{p'\left(d-\frac{1}{2}\right)-d}\:\:&\text{ if } \frac{2d}{d-1} < p' <\infty\:,
	\end{cases}
\nonumber\\
&\lesssim & \begin{cases}
      N^{\frac{p'd}{2}}\:\:&\text{ if } 1 \le p' < \frac{2d}{d-1}\:,\\
	 N^{\frac{p'd}{2}+}\:\:&\text{ if } p' = \frac{2d}{d-1}\:,\\
	 N^{p'\left(d-\frac{1}{2}\right)-d}\:\:&\text{ if } \frac{2d}{d-1} < p' <\infty\:.
	\end{cases}
\end{eqnarray}

We now shift our attention to $\theta \in \left[\frac{\pi}{2},\pi\right]$. By invoking part (\ref{part3}) of Theorem \ref{thm_oscillatory_zn} for integers $m \ge \max\left\{\beta +1,\frac{d+2(1-\beta)}{4} \right\}$ into (\ref{thm_talbot_pf_eq3}), we get
\begin{eqnarray*}
\left|H_{N,t}(\theta)\right| &\le& \left|\sum_{n=N+1}^{2N} e^{itn(n+\alpha+\beta+1)}\:e^{in\theta}\lambda^+_n(\theta)\right|\\
&&+ \left|\sum_{n=N+1}^{2N} e^{itn(n+\alpha+\beta+1)}\:e^{-in\theta}\lambda^-_n(\theta)\right|\\
&&+\left|\sum_{n=N+1}^{2N} e^{itn(n+\alpha+\beta+1)}\:E_2(\theta, n)\right|\:.
\end{eqnarray*}
Then by the pointwise bounds of $\left|E_2(\theta, n)\right|$, we have
\begin{equation} \label{thm_talbot_pf_eq9}
\left|\sum_{n=N+1}^{2N} e^{itn(n+\alpha+\beta+1)}\:E_2(\theta, n)\right| \lesssim N^{\left(\beta + \frac{d}{2}+1\right)-m}\:.
\end{equation}
To estimate the other two Weyl sums, we consider two separate cases. First when $\theta \in \left[\frac{\pi}{2},\pi -\frac{1}{N}\right]$, combining the estimates of $|\lambda^\pm_n(\theta)|$ and $|\lambda^\pm_n(\theta)-\lambda^\pm_{n-1}(\theta)|$ (provided in part (\ref{part3}) of Theorem \ref{thm_oscillatory_zn}) and Lemma \ref{Weyl_estimate}, we get
\begin{equation} \label{thm_talbot_pf_eq10}
\left|\sum_{n=N+1}^{2N} e^{in^2t+in\left(t(\alpha+\beta+1)\pm \theta\right)}\:\lambda^\pm_n(\theta)\right| \lesssim_t \frac{N^{\frac{d}{2}}}{(\pi-\theta)^{\beta+\frac{1}{2}}} \lesssim \frac{N^{\beta+\frac{d+1}{2}}}{\langle N(\pi-\theta)\rangle^{\beta+\frac{1}{2}}}\:.
\end{equation}
In the second case, when $\theta \in \left[\pi-\frac{1}{N},\pi\right]$, again combining the estimates of $|\lambda^\pm_n(\theta)|$ and $|\lambda^\pm_n(\theta)-\lambda^\pm_{n-1}(\theta)|$ (provided in part (\ref{part3}) of Theorem \ref{thm_oscillatory_zn}) and Lemma \ref{Weyl_estimate}, we get
\begin{equation} \label{thm_talbot_pf_eq11}
\left|\sum_{n=N+1}^{2N} e^{in^2t+in\left(t(\alpha+\beta+1)\pm \theta\right)}\:\lambda^\pm_n(\theta)\right| \lesssim_t N^{\beta+\frac{d+1}{2}}\lesssim \frac{N^{\beta+\frac{d+1}{2}}}{\langle N(\pi-\theta)\rangle^{\beta+\frac{1}{2}}}\:.
\end{equation}
Thus as $m \ge \beta +1$, by (\ref{thm_talbot_pf_eq9})-(\ref{thm_talbot_pf_eq11}), we get 
\begin{equation} \label{kernel_estimate2}
\left|H_{N,t}(\theta)\right| \lesssim_t \frac{N^{\beta+\frac{d+1}{2}}}{\langle N(\pi-\theta)\rangle^{\beta+\frac{1}{2}}}\:,\:\:\theta \in \left[\frac{\pi}{2},\pi\right]\:.
\end{equation}
Then by (\ref{kernel_estimate1}) and (\ref{kernel_estimate2}), we get that
\begin{equation*}
\left\|H_{N,t}\right\|_{L^\infty(\X)} \lesssim_t N^{d-\frac{1}{2}}\:,
\end{equation*}
which upon plugging in (\ref{thm_talbot_pf_eq4}) along with the hypothesis $\left\|P_Nf\right\|_{L^1(\X)} \lesssim N^{-\left(\frac{d}{2}+s\right)}$ yields,
\begin{equation}\label{L1_case}
\left\|P_Nu(\cdot,t)\right\|_{L^\infty(\X)} \lesssim_t N^{\frac{d-1}{2}-s}\:.
\end{equation}
For $p' \ne \infty$, the estimate (\ref{kernel_estimate2}) and the formula for the density (\ref{density}) yield,
\begin{equation*}
\int_{\frac{\pi}{2}}^\pi \left|H_{N,t}(\theta)\right|^{p'}\mathcal{A}(\theta)\:d\theta \lesssim_t N^{p'\left(\beta+\frac{d+1}{2}\right)}\int_{\frac{\pi}{2}}^\pi \frac{{\left(\sin \frac{\theta}{2}\right)}^{M_1} {(\sin \theta)}^{M_2}}{\langle N(\pi-\theta)\rangle^{p'\left(\beta+\frac{1}{2}\right)}}\:d\theta\:.
\end{equation*}
Now by the change of variable, $\eta=\pi-\theta$, we get,
\begin{eqnarray*}
&&\int_{\frac{\pi}{2}}^\pi \left|H_{N,t}(\theta)\right|^{p'}\mathcal{A}(\theta)\:d\theta \\
&\lesssim_t & N^{p'\left(\beta+\frac{d+1}{2}\right)} \int_0^{\frac{\pi}{2}} \frac{\eta^{M_2}}{\langle N\eta \rangle^{p'\left(\beta +\frac{1}{2}\right)}}\:d\eta \\
&=& N^{p'\left(\beta+\frac{d+1}{2}\right)} \int_0^{\frac{1}{N}} \frac{\eta^{M_2}}{\langle N\eta \rangle^{p'\left(\beta +\frac{1}{2}\right)}}\:d\eta \:+\:N^{p'\left(\beta+\frac{d+1}{2}\right)} \int_{\frac{1}{N}}^{\frac{\pi}{2}} \frac{\eta^{M_2}}{\langle N\eta \rangle^{p'\left(\beta +\frac{1}{2}\right)}}\:d\eta \\
& \lesssim & N^{p'\left(\beta+\frac{d+1}{2}\right)-(M_2+1)}\:+\:N^{\frac{p'd}{2}}\int_{\frac{1}{N}}^{\frac{\pi}{2}} \eta^{M_2-p'\left(\beta+\frac{1}{2}\right)}\:d\eta\:.
\end{eqnarray*}
A simple calculation now yields,
\begin{eqnarray} \label{thm_talbot_pf_eq12}
\int_{\frac{\pi}{2}}^\pi \left|H_{N,t}(\theta)\right|^{p'}\mathcal{A}(\theta)\:d\theta &\lesssim_t & N^{p'\left(\beta+\frac{d+1}{2}\right)-(M_2+1)} + \begin{cases}
      N^{\frac{p'd}{2}}\:\:&\text{ if } 1 \le p' < \frac{2(M_2+1)}{2\beta +1}\:,\\
	 N^{\frac{p'd}{2}+}\:\:&\text{ if } p' =\frac{2(M_2+1)}{2\beta +1}\:,\\
	 N^{p'\left(\beta+\frac{d+1}{2}\right)-(M_2+1)} \:\:&\text{ if }  \frac{2(M_2+1)}{2\beta +1} <p' <\infty\:,
	\end{cases}
\nonumber\\
&\lesssim & \begin{cases}
      N^{\frac{p'd}{2}}\:\:&\text{ if } 1 \le p' < \frac{2(M_2+1)}{2\beta +1}\:,\\
	 N^{\frac{p'd}{2}+}\:\:&\text{ if } p' =\frac{2(M_2+1)}{2\beta +1}\:,\\
	 N^{p'\left(\beta+\frac{d+1}{2}\right)-(M_2+1)} \:\:&\text{ if }  \frac{2(M_2+1)}{2\beta +1} <p' <\infty\:.
	\end{cases}
\end{eqnarray}
Thus comparing (\ref{thm_talbot_pf_eq8}) and (\ref{thm_talbot_pf_eq12}) with an eye on Tables \ref{table:1} and \ref{table:2} yield,
\begin{equation*} 
\left\|H_{N,t}\right\|_{L^{p'}(\X)} \lesssim_t \begin{cases}
      N^{\frac{d}{2}}\:\:&\text{ if } 1 \le p' < \frac{2d}{d-1}\:,\\
	 N^{\frac{d}{2}+}\:\:&\text{ if } p' = \frac{2d}{d-1}\:,\\
	 N^{d\left(1-\frac{1}{p'}\right)-\frac{1}{2}}\:\:&\text{ if }  \frac{2d}{d-1}<p'<\infty\:.
	\end{cases}
\end{equation*}
Then the above estimate along with (\ref{thm_talbot_pf_eq4}) and the hypothesis $\|P_Nf\|_{L^p(\X)} \lesssim N^{-\left(\frac{d}{2}+s\right)}$ yields for almost all $t$,
\begin{equation*}
\left\|P_Nu(\cdot,t)\right\|_{L^\infty(\X)} \lesssim_t \begin{cases}
	 N^{\frac{d}{p}-\frac{d+1}{2}-s}\:\:&\text{ if } 1<p < \frac{2d}{d+1}\:.\\
	 	 N^{-s+}\:\:&\text{ if } p=\frac{2d}{d+1}\:,\\
	 	 N^{-s}\:\:&\text{ if } \frac{2d}{d+1} < p \le \infty\:,
	\end{cases}
\end{equation*}
which after combining with the $L^1$ case (\ref{L1_case}), we get
\begin{equation*}
\left\|P_Nu(\cdot,t)\right\|_{L^\infty(\X)} \lesssim_t \begin{cases}
	 N^{\frac{d}{p}-\frac{d+1}{2}-s}\:\:&\text{ if } 1\le p < \frac{2d}{d+1}\:.\\
	 	 N^{-s+}\:\:&\text{ if } p=\frac{2d}{d+1}\:,\\
	 	 N^{-s}\:\:&\text{ if } \frac{2d}{d+1} < p \le \infty\:,
	\end{cases}
\end{equation*}
This yields (\ref{thm_talbot_pf_eq2}) and completes the proof of part (\ref{talbot_part1}) of Theorem \ref{thm_talbot}.

\medskip

To get part (\ref{talbot_part2}), we note that if $f \in C^{\gamma'}(\X)$ then so does its real and imaginary parts. Then the result follows by applying Corollary \ref{graph_corollary} on the real and imaginary parts.
\end{proof}

\begin{proof}[Proof of Corollary \ref{talbot_corollary}]
Corollary \ref{talbot_corollary} follows at once from Theorem \ref{thm_talbot} by observing that we must have $$s=\frac{d}{p}-\frac{d}{2}=\frac{d(d+1)}{2d}-\frac{d}{2}=\frac{1}{2}\:.$$
\end{proof}

\section{Concluding remarks}\label{sec8}
\begin{enumerate}
\item In Theorem \ref{thm_cesaro}, we have worked under the assumption that $\delta \ge \delta_*$. The reason behind such a choice was the crucial estimate (\ref{cesaro_kernel_estimate}). But since one has the Ces\`aro summability for H\"older functions as soon as $\delta>(d-1)/2$ \cite{Bonami}, it would be interesting to study whether the index $\delta$ in Theorem \ref{thm_cesaro} can be chosen in $\left(\frac{d-1}{2},\delta_*\right)$.

\medskip
  
\item In Theorem \ref{thm_talbot} and Corollary \ref{talbot_corollary}, we have only studied upper bounds on the upper Minkowski dimension. It is thus natural to also ask about lower bounds on the fractal dimension. In the case of $\Sb^d$ itself, some results are known only for $\Sb^2$ \cite[Theorems 1.1, 1.3, 3.8]{MathZ}. For higher dimensional manifolds, this seems to be a challenging task however as it may require stronger Strichartz estimates than what are currently known. For the interested reader, we refer them to compare the Strichartz estimates obtained for general compact manifolds in \cite{Burq} with for the tori $\mathbb T^d$ in \cite{Bourgain}.

\medskip

\item Finally, it would be interesting to see possible generalizations of the results appearing in this article to higher rank Riemannian symmetric spaces of compact type.
\end{enumerate}

\section*{Acknowledgements}  The author is supported by the Institute Post Doctoral Fellowship of Indian Institute of Technology, Bombay.

\bibliographystyle{amsplain}

\end{document}